\documentclass[10pt,twoside]{article}
\usepackage{amsmath, amsthm, amsfonts, amssymb}
\usepackage{mathrsfs}
\usepackage[misc]{ifsym}
\usepackage[colorlinks,linkcolor=blue,anchorcolor=blue,citecolor=blue,urlcolor=black]{hyperref}
\usepackage{graphicx}
\usepackage{subfigure}
\usepackage{tikz}
\usepackage{enumerate}
\usepackage{anysize}
\usepackage{setspace}
\usepackage{float}
\def\d{\textup{d}}
\usepackage{color}
\usepackage{chngcntr}
\usepackage[numbers,sort&compress]{natbib}
\everymath{\displaystyle}
\allowdisplaybreaks

\newfam\msbfam
\font\tenmsb=msbm5    \textfont\msbfam=\tenmsb \font\sevenmsb=msbm5
\scriptfont\msbfam=\sevenmsb \font\fivemsb=msbm5
\scriptscriptfont\msbfam=\fivemsb

\newfam\bigfam
\font\tenbig=msbm5 scaled \magstep2   \textfont\bigfam=\tenbig
\font\sevenbig=msbm7 scaled \magstep2 \scriptfont\bigfam=\sevenbig
\font\fivebig=msbm5 scaled \magstep2
\scriptscriptfont\bigfam=\fivebig

\newtheorem{theorem}{Theorem}[section]
\newtheorem{lemma}[theorem]{Lemma}
\newtheorem{corollary}[theorem]{Corollary}

\theoremstyle{definition}
\newtheorem{remark}[theorem]{Remark}

\newtheorem{definition}[theorem]{Definition}

\begin{document}
	\title{\bf  Multilinear Strongly Singular Integral Operators with Generalized Kernels on RD-Spaces}
	\author{\bf Kang Chen, Yan Lin$^*$ and Shuhui  Yang}
	\renewcommand{\thefootnote}{}
	\date{}
	\maketitle
	\footnotetext{This work was partially supported by the National Natural Science Foundation of China (Grant Nos. 12471090)  and China Postdoctoral Science Foundation (Grant Nos. 2024M760238).}
	\footnotetext{$^*$Corresponding author, E-mail:linyan@cumtb.edu.cn}
	\begin{minipage}{13.5cm}
		{\bf Abstract}
		\quad 
		In this article, we introduce a class of multilinear strongly singular integral operators with generalized kernels on the RD-space. The boundedness of these operators on weighted Lebesgue spaces is established. Moreover, two types of endpoint estimates and their boundedness on generalized weighted Morrey spaces are obtained. Our results further generalize the relevant conclusions on generalized kernels in Euclidean spaces. Moreover, the weak-type results on weighted Lebesgue spaces are brand new even in the situation of Euclidean spaces. In addition, when the generalized kernels degenerate into classical kernels, our research results also extend the relevant known results. It is worth mentioning that our RD spaces are more general than theirs.

	\end{minipage}
	
	\section{Introduction}\label{sec1}
	
	\qquad In the 1950s, Calder\'{o}n and Zygmund introduced the classical Calder\'{o}n--Zygmund operator and established the foundational theory of Calder\'{o}n--Zygmund operators (see, for instance, \cite{A1952,A1954,A1957}). Since then, the theory of singular integrals has rapidly developed into a central component of harmonic analysis, with wide-ranging applications in fields such as partial differential equations, nonlinear analysis, and geometric measure theory (see, for instance, \cite{L2007,E1970,L1977,M1990,G1991,F1983,C1972,L2014}). However, not all singular integral operators are Calder\'{o}n--Zygmund  operators, and some operators have stronger singularities. With the development of the theory of singular integrals, more and more scholars have begun to study strongly singular integral operators with stronger singularities (see, for instance, \cite{J2005, J2006, J20062}). The notion of strongly singular integral operators was initially introduced through multiplier operators, defined as $$(T_{\alpha,\beta}f)^{\wedge}(\xi)=\theta(\xi)\frac{e^{i|\xi|^{\alpha}}}{|\xi|^{\beta}}\hat{f}(\xi),$$ where $0<\alpha<1$, $0<\beta<n\alpha/2$ and $\theta(\xi)$ is a standard cutoff function. Fefferman referred to these operators as weakly-strongly singular Calder\'{o}n--Zygmund operators in \cite{C1970}. Stein, Wainger,  Hirschman, and Fefferman established the boundedness of these operators in various function spaces respectively (see, for instance, \cite{I1959,E1967,S1965}). As convolution-type operators, their behavior depends on the size and smoothness conditions of the kernel function. As research progressed, scholars began to explore non-standard kernels (see, for instance, \cite{M1987,R1982,F2003}), leading to the introduction of non-convolution-type strongly singular integral operators by Alvarez and Milman in \cite{J1986}. Compared to classical Calder\'{o}n--Zygmund operators, these operators exhibit stronger singularities in their kernels.\par
	With the advancement of harmonic analysis, attention gradually shifted towards the study of multilinear operators.  Coifman and Meyer extended the theory of linear Calder\'{o}n--Zygmund operators to the multilinear case, defining multilinear Calder\'{o}n--Zygmund operators and applying them to areas such as wavelet analysis (see, for instance, \cite{R1975,R1978,R19782}). Their work stimulated the development of multilinear operator theory, opening new avenues for modern applications in harmonic analysis and drawing increasing interest in the study of multilinear singular integrals. Lin \cite{Y2020} defined a class of multilinear strongly singular Calder\'{o}n--Zygmund operators, which, unlike typical multilinear Calder\'{o}n--Zygmund operators, do not require any size condition on their kernel and exhibits stronger singularities near the diagonal. The boundedness of these operators in various function spaces was established in \cite{Y2019,Y20192}. Subsequently, Yang and Lin \cite{S2021} extended this research by replacing the classical kernel conditions with a generalized kernel condition of integral type, defining generalized strongly singular integral operators and establishing their boundedness on weighted Lebesgue spaces, variable exponeat Lebesgue spaces, and BMO spaces. To match the multiple weights, Wei et al.  defined a new class of strongly singular integral operators with generalized kernels in \cite{B2023}, which are defined as follows.
	\begin{definition}\label{definition1.1}{\rm(see \cite{B2023})}
		$T$ is given by $$T(f_1,...,f_m)(x)=\int_{(\mathbb{R}^n)^m}K(x,y_1,...,y_m)\prod\limits_{j=1}^mf_j(y_j)\d {\vec{y}},$$ 
		where $\{f_j\}_{j=1}^{\infty} \in C_c^{\infty}(\mathbb{R}^n)$ and $x\notin \cap_{j=1}^m {\rm{supp}} \{f_j\}$. $T$ is called an \textit{$m$-linear strongly singular integral operator with generalized kernel} on $\mathbb{R}^n$ if  the following conditions are satisfied.
		\begin{itemize}
			\item [\rm (i)] Let $m\in \mathbb{N}_+$ and $K(y_0,y_1,...,y_m)$ is a function defined away from the diagonal $y_0=y_1=...=y_m$ in $(\mathbb{R}^n)^{m+1}$. For some $\varepsilon>0$, $0 < \alpha \leq 1$, for any $k\in \mathbb{N}_+$,
			\begin{align*}
				& \quad\left(\int_{2^k\sqrt{m}\left|x-x'\right|^t\leq \left|\vec{y}-\vec{x}\right|\le 2^{k+1}\sqrt{m}\left|x-x'\right|^t}\left|K(x,y_1,...,y_m)-K(x',y_1,...,y_m)\right|^{p_0}\d {\vec{y}} \right)^{\frac{1}{p_0}}            \\
				&\leq \displaystyle  C|x-x'|^{\varepsilon-t\left(\frac{nm}{p_0'}+\frac{\varepsilon}{\alpha}\right)}2^{-k\left(\frac{nm}{p_0'}+\frac{\varepsilon}{\alpha}\right)},
			\end{align*}
			where $\vec{y}:=(y_1,...,y_m)$, $\vec{x}:=(x,...,x)$, $t=1$ when $|x-x'| \geq 1$, $t=\alpha$ when $|x-x'| < 1$, $p_0$ is a fixed  positive number that satisfies $1/p_0+ 1/p_0'=1$ and $\displaystyle  1<p_0<\infty $.
			
			\item [\rm (ii)]   For some given numbers $1\le r_1,..., r_m<\infty $ with $1/r=1/r_1+\cdots+1/r_m$, $T$ is bounded from $L^{r_1}\left(\mathbb{R}^n\right)\times \cdots \times L^{r_m}\left(\mathbb{R}^n\right)$ to $L^{r,\infty}\left(\mathbb{R}^n\right)$.

			\item [\rm (iii)]  For some given numbers $1\le l_1,..., l_m<\infty $ with $1/l=1/l_1+\cdots+1/l_m$, $T$ is bounded from $L^{l_1}(\mathbb{R}^n)\times \cdots \times L^{l_m}(\mathbb{R}^n)$ to $L^{q,\infty}(\mathbb{R}^n)$, where $0<l/q\leq \alpha$.	
		\end{itemize}
	\end{definition}	
Wei et al. \cite{B2023} established the boundedness of these operators on weighted Lebesgue spaces, particularly in the context of multiple weights, and further demonstrated the boundedness of commutators generated by BMO functions. For some other results, one can see \cite{B2023, S2022, S2023} and related references. \par  
	The development of \text{\rm RD}-spaces originated from the homogeneous spaces defined by Coifman and Weiss (see, for instance, \cite{R1971, R1977}), which were extended to \text{\rm RD}-spaces by Han et al. (see, for instance, \cite{ Y2006, Y2008}). In \cite{L20142}, Grafakos et al. demonstrated that dyadic cubes, covering lemmas, and the Calder\'{o}n--Zygmund decomposition remain applicable in \text{\rm RD}-spaces, positioning \text{\rm RD}-spaces as an abstract generalization of classical Euclidean spaces. In recent years, \text{\rm RD}-spaces have found widespread applications in fields such as analysis and partial differential equations (see, for instance, \cite{P2010, P2011, D2010, D2011, D20112}). Grafakos and Torres \cite{L2002} established the theory of multilinear Calder\'{o}n--Zygmund operators in \text{\rm RD}-spaces, and building on their work, Grafakos et al. \cite{L20142} extended the classical Euclidean Calder\'{o}n--Zygmund theory to \text{\rm RD}-spaces and defined multilinear Calder\'{o}n--Zygmund operators in this setting. To further explore Calder\'{o}n--Zygmund theory in \text{\rm RD}-spaces, Li and Wu \cite{W2024} introduced strongly singular integral operators on RD-spaces, where the conditions of kernels are pointwise estimates. They established the boundedness of the operators on weighted Lebesgue spaces and on BMO spaces. Moreover, they also established the boundedness of multilinear commutators  with BMO functions.
	\begin{definition}\label{definition1.2}{\rm(see \cite{L20142})}
		Let $(\mathcal{X},d)$ be a metric space,  for any $x \in \mathcal{X}$ and $r\in (0,+\infty)$,  the ball $B(x,r):=\{y\in \mathcal{X}:d(x,y)<r\}$. Suppose that $\mu$ is a regular Borel measure defined on a $\sigma$-algebra, which contains all Borel sets induced by the open balls $\{B(x,r):x\in \mathcal{X},r>0\}$, and that $0<\mu\left[B(x,r)\right]<\infty $ for all $x\in \mathcal{X}$ and $r>0$.

		\begin{itemize}
			\item [\rm (i)] The triple ($\mathcal{X},d,\mu $) is called a space of homogeneous type if $\mu$ satisfies the doubling condition: there exists a constant $C_1\in [1,\infty)$ such that, for all $x\in \mathcal{X}$ and $r>0$,
			\begin{equation}\label{1.1}
				\mu[B(x,2r)]\leq  C_1\mu[B(x,r)].
			\end{equation}
			
			\item [\rm (ii)]  The triple ($\mathcal{X},d,\mu $) is called an \text{\rm RD}-space if it is a space of homogeneous type and $\mu$ satisfies the reverse doubling condition: there exists a $\kappa \in (0,\infty)$ and  $C_2\in (0,1]$ such that, for all $x\in \mathcal{X}$, $0<r < \rm 2diam(\mathcal{X})$ and $1\leq \lambda< \rm 2diam(\mathcal{X})/r$,
			\begin{equation}\label{1.2}
				C_2\lambda^{\kappa}\mu[B(x,r)]\leq  \mu[B(x,\lambda r)],
			\end{equation}
			where $${\rm diam}(\mathcal{X}):=\sup\limits_{x,y\in \mathcal{X}}d(x,y).$$
		\end{itemize}
	\end{definition}

	\begin{remark}\label{Remark1.1}
		\begin{itemize}
			\item [\rm (i)] For an \text{\rm RD}-space $\mathcal{X}$, by \eqref{1.1} and \eqref{1.2}, there exist $C_3\in [1,\infty)$ and $n\in (0,\infty)$ such that, for all $x\in \mathcal{X},r>0$, and $\lambda\geq 1$, \begin{equation}\label{1.3}
				C_2\lambda^{\kappa}\mu[B(x,r)]\leq  \mu[B(x,\lambda r)]\leq C_3\lambda^{n}\mu[B(x,r)].
			\end{equation}
			Indeed, we can choose $C_3:=C_1$ and $n:= \log_2C_1$. In some sense, $n$ measures the “upper dimension” of $\mathcal{X}$. We obviously have $n\geq \kappa$ and call such  \text{\rm RD}-space  $(\kappa,n)$-space.
			
			\item [\rm (ii)]  We will also assume that there exists a positive nondecreasing function $\varphi$ defined on $[0,\infty)$ such that for all $x\in \mathcal{X}$ and $r>0$, \begin{equation}\label{1.4}
				\mu[B(x,r)]\thicksim \varphi(r). 
			\end{equation}
			Notice that \eqref{1.3} and \eqref{1.4} imply that \begin{equation}\label{1.5}
				r^{\kappa} \lesssim \varphi(r) \lesssim r^n \quad \text{if\quad } r \geq 1,
			\end{equation}
			and \begin{equation}\label{1.6}
				r^n \lesssim \varphi(r) \lesssim r^{\kappa} \quad \text{if\quad } 0< r< 1.
			\end{equation}
			
			\item [\rm (iii)] For a given homogeneous space $(\mathcal{X},d,\mu)$, Nakai and Yabuta proved in \cite{E1997} that if $\mu(\mathcal{X}) < \infty,$ then  there is a positive constant $R_0$ that makes $\mathcal{X}=B(x,R_0)$ for all $x\in \mathcal{X}.$ In this article, we always assume $(\mathcal{X},d,\mu)$ is an \text{\rm RD}-space and $\mu(\mathcal{X}) = \infty$.
		\end{itemize}
	\end{remark}

	\begin{definition}\label{definition1.6}{\rm(see \cite{L20142})}
		For any $\eta \in (0,1]$, we use $C^{\eta}(\mathcal{X})$ to represent all functions on $\mathcal{X}$ that satisfy the following conditions$$\|f\|_{C^{\eta}(\mathcal{X})}:=\sup\limits_{x\ne y}\frac{|f(x)-f(y)|}{d(x,y)^{\eta}}<\infty,$$ 
		and denote all functions in $C^{\eta}(\mathcal{X})$ that have a bounded support set as $C_b^{\eta}(\mathcal{X})$. The notation $[C_b^{\eta}(\mathcal{X})]'$ denotes the space of all continuous linear functionals on $C_b^{\eta}(\mathcal{X})$, which is the dual space of $C_b^{\eta}(\mathcal{X})$, refer to \cite{L20142} for more details.
		
	\end{definition}	
	\begin{definition}\label{definition1.3}{\rm(see \cite{L20142})}
		Given $m\in \mathbb{N}_{+}$, set $$\Omega_m:=\mathcal{X}^{m+1}\backslash \{(y_0,...,y_m):y_0=...=y_m\}.$$
		Suppose that $K : \Omega_m \rightarrow \mathbb{C}$ is locally integrable. The function $K$ is called a Calder\'{o}n--Zygmund kernel if there exist
		constants $C_K\in (0,+\infty)$ and $\delta \in (0,1]$ such that, for all $(y_0,y_1,...,y_m) \in \Omega_m$,$$|K(y_0,y_1,...,y_m)|\leq C_K\frac{1}{\left[\sum_{k=1}^mV(y_0,y_k)\right]^m},$$
		and that, for all $k\in \{1,...,m\}$,
		\begin{align*}
			\quad|K(y_0,y_1,...,y_k,...,y_m)-K(y_0,y_1,...,y_k',...,y_m)|
			\leq C_K\left[\frac{d(y_k,y_k')}{\max\limits_{0\leq k\leq m}d(y_0,y_k)}\right]^{\delta}\frac{1}{\left[\sum_{k=1}^mV(y_0,y_k)\right]^m},
		\end{align*}
		with $V(y_0,y_k):=\mu\left[B\left(y_0,d(y_0,y_k)\right)\right]$, whenever $d(y_k,y_k') \leq \max\limits_{0\leq k \leq m} d(y_0,y_k)/2$.
	\end{definition}
	Inspired by the above-mentioned studies, this article defines a class of strongly singular integral operators with generalized kernels in RD-spaces, further exploring the Calder\'{o}n--Zygmund theory on RD-spaces.
	
	\begin{definition}\label{definition1.4}
		Suppose that $K : \Omega_m \rightarrow \mathbb{C}$ is locally integrable. The function $K$ is called a \textit{strongly singular generalized kernel} if the following conditions are satisfied. For some $\varepsilon>0$, $0 < \alpha \leq 1$, for any $k\in \mathbb{N}_+$,
		\begin{align*}
			& \displaystyle \quad\left(\int_{2^kd(x,x')^t\leq \rho(\vec{y},\vec{x})\le 2^{k+1}d(x,x')^t}|K(x,y_1,...,y_m)-K(x',y_1,...,y_m)|^{p_0}\d {\mathbf{\overrightarrow{\mu(y)}}} \right)^{\frac{1}{p_0}}            \\
			&\displaystyle  \leq \displaystyle  Cd(x,x')^{\varepsilon-t\left(\frac{\beta m}{p_0'}+\frac{\varepsilon}{\alpha}\right)}2^{-k\left(\frac{nm}{p_0'}+\frac{\varepsilon}{\alpha}\right)},
		\end{align*}
		where $\vec{y}:=(y_1,... ,y_m)$, $\vec{x}:=(x,...,x)$, $\rho(\vec{y},\vec{x}):=\max\limits_{1\leq j\leq m}d(y_j,x)$,  $t=1$, $\beta = n$ when $d(x,x') \geq 1$, $t=\alpha$, $\beta =\kappa$ when $d(x,x') < 1$, $p_0$ is a fixed  positive number that satisfies $1/p_0+ 1/p_0'=1$ and $\displaystyle  1<p_0<\infty $.
	\end{definition}

	\begin{definition}\label{definition1.5}
		
		Let $\eta\in (0,1]$ and $T$ be an $m$-linear operator 
		$$T: \overbrace{C_b^{\eta}(\mathcal{X})\times\cdots \times C_b^{\eta}(\mathcal{X})}^{\text{$m$ times}}\rightarrow (C_b^{\eta}(\mathcal{X}))',$$
		such that for all $f_1,... ,f_m\in C_b^{\eta}(\mathcal{X})$ and $x \notin \cap_{j=1}^m $supp$\{f_j\} $,$$T\left(f_1,...,f_m\right)(x)=\int_{\mathcal{X}^m}K(x,y_1,...,y_m)\prod\limits_{j=1}^mf_j(y_j)\d {\mathbf{\overrightarrow{\mu(y)}}}.$$
		$T$ is called an \textit{$m$-linear strongly singular integral operator with generalized kernel on \text{\rm RD}-space} (for short, $m$-GSSIO) if  the following conditions are satisfied.
		\begin{itemize}
			\item [\rm (i)] The kernel function $K$ is the strongly singular generalized kernel given in Definition \ref{definition1.4}.
			
			\item [\rm (ii)]   For some given numbers $1\le r_1,..., r_m<\infty $ with $1/r=1/r_1+\cdots+1/r_m$, $T$ is bounded from $L^{r_1}(\mathcal{X})\times \cdots \times L^{r_m}(\mathcal{X})$ to $L^{r,\infty}(\mathcal{X})$.

			\item [\rm (iii)]  For some given numbers $1\le l_1,..., l_m<\infty $ with $1/l=1/l_1+\cdots+1/l_m$, $T$ is bounded from $L^{l_1}(\mathcal{X})\times \cdots \times L^{l_m}(\mathcal{X})$ to $L^{q,\infty}(\mathcal{X})$, where $0<l/q\leq \alpha\kappa/n$.
			
		\end{itemize}
		
	\end{definition}
	
	\begin{remark}
		\begin{itemize}
			\item[\rm (i)] Obviously, if we take $\mathcal{X}$ as an $n$-dimensional Euclidean space $\mathbb{R}^n$, then the kernel condition here will revert to the case provided in Definition \ref{definition1.1} (i). Under the same remaining conditions, $T$ is the strongly singular integral operator with a generalized kernel as defined in \cite{B2023}.
			\item[\rm (ii)] When $\kappa=n$, the generalized kernel condition given in Definition \ref{definition1.4} can easily imply the classical kernel defined in \cite{W2024} by Li and Wu. In other words, the strongly singular integral operator on RD-spaces in \cite{W2024} is a special case of the operator we defined.
		\end{itemize}
	\end{remark}
	The structure of this article is arranged as follows. In Section \ref{sec2}, we provide pointwise estimates for the sharp maximal function, and based on these, we prove the strong and weak boundedness of \( T \) on weighted Lebesgue spaces. This not only generalizes the strong-type results in \cite{W2024}, but also yields new weak-type results. In Section \ref{sec3}, we present two types of endpoint estimates, namely, \( L^{\infty}(\mathcal{X}) \times \cdots \times L^{\infty}(\mathcal{X}) \rightarrow BMO(\mathcal{X}) \) and \( BMO(\mathcal{X}) \times \cdots \times BMO(\mathcal{X}) \rightarrow BMO(\mathcal{X}) \). In Section \ref{sec4}, we establish the boundedness of \( T \) on generalized weighted Morrey spaces on \text{\rm RD}-spaces.
	
	Throughout the article, for any $1\leq p<\infty$, $p'$ satisfies the condition $1/p+1/p'=1$. The notation $f\lesssim g$ represents $f\leq Cg$ and $f\gtrsim g$ represents $f \geq Cg$, where $C$ does not depend on $f$ and $g$. If $f\lesssim g \lesssim f$, we then write  $f\thicksim g$.

	\section{The Boundedness on Weighted Lebesgue Spaces}\label{sec2}
	
	\qquad In this section, we first recall some definitions and necessary lemmas that will be used later in subsection 2.1, then establish pointwise estimates of sharp maximal functions in subsection 2.2, and finally use these lemmas to establish the boundedness of strongly singular integral operators with generalized kernels on weighted Lebesgue spaces in subsection 2.3.
	
	\subsection{Definitions and Lemmas}
	
	\begin{definition}\label{definition2.1}{\rm(see \cite{L20142})}
		Let $1< p <\infty $ and $1/p + 1/p'= 1$, a non-negative locally integrable function $w$ on $\mathcal{X}$ belongs to the Muckenhoupt class $A_p$ if $$\sup\limits_{B}\left(\frac{1}{\mu(B)}\int_{B}w(x)\d{\mu(x)}\right)\left(\frac{1}{\mu(B)}\int_{B}w(x)^{1-p'}\d {\mu(x)}\right)^{p-1} < \infty,$$
		where the supremum is taken over all balls $B\subset\mathcal{X}$. When $p = 1 $, a weight $w$ belongs to the class $A_1$ if $$\sup\limits_{B}\left(\frac{1}{\mu(B)}\int_{B}w(x)\d{\mu(x)}\right)\left(\inf\limits_{B}w(x)\right)^{-1} < \infty.$$
	\end{definition}
	Set $$A_{\infty}:= \bigcup\limits_{1\leq p<\infty}A_p.$$
	For more details on the $A_p$ weights on spaces of homogeneous type, see for instance \cite{J1989}.
	
	\begin{definition}\label{definition2.2}{\rm(see \cite{L20142})}
		Let $\vec{p}:= (p_1,..., p_m)$ satisfies $1\leq p_1,  ... ,p_m<\infty$ and $1/p=1/p_1+\cdots+1/p_m$. Suppose  $\vec{w}:= (w_1,..., w_m)$ with every $w_j$ a weight. Set $\nu_{\vec{w}} := \prod\limits_{j=1}^mw_j^{p/p_j}$. We say that ${\vec{w}} $ satisfies the $A_{\vec{p}}$ condition if $$\sup\limits_{B}\left(\frac{1}{\mu(B)}\int_{B}\nu_{\vec{w}}(x)\d{\mu(x)}\right)^{1/p}\prod\limits_{j=1}^m\left(\frac{1}{\mu(B)}\int_{B}w_j(x)^{1-p_j'}\d {\mu(x)}\right)^{1/p_j'} < \infty,$$
		where, when $p_j=1$, $$\left(\frac{1}{\mu(B)}\int_{B}w_j(x)^{1-p_j'}\d {\mu(x)}\right)^{1/p_j'}$$
		is understood as $\left(\inf\limits_{B}w_j\right)^{-1}$. 
	\end{definition}

	 We now recall the $L^p$ and weighted $L^p$ spaces on \text{\rm RD}-spaces.
	\begin{definition}\label{definition2.3}{\rm(see \cite{L20142})}
		Let $f$ be a measurable function on $\mathcal{X}$. For $p>0$, the $L^p(\mathcal{X})$ space is defined by $$L^p(\mathcal{X}):=\{f: \|f\|_{L^p(\mathcal{X})}<\infty\},$$
		where $$\left\|f\right\|_{L^p(\mathcal{X})}:=\left(\int_{\mathcal{X}}|f(x)|^p\d{\mu(x)}\right)^{1/p}.$$
		The $L^{p,\infty}(\mathcal{X})$ space is defined by $$L^{p,\infty}(\mathcal{X}):=\{f: \|f\|_{L^{p,\infty}(\mathcal{X})}<\infty\},$$
		where$$\|f\|_{L^{p,\infty}(\mathcal{X})}:=\sup\limits_{\lambda>0}\lambda\left(\mu\left(\{x\in \mathcal{X} : |f(x)|>\lambda\}\right)\right)^{1/p}.$$
	\end{definition}
	
	\begin{remark}
		Let $L_{loc}^1(\mathcal{X})$ to denote the space of all locally integrable functions on $\mathcal{X}$. Furthermore, for any $q \in (0, \infty)$, $$L_{loc}^q(\mathcal{X}):=\{f: |f|^q\in L_{loc}^1(\mathcal{X})\}.$$
		
	\end{remark}
	\begin{definition}\label{definition2.4}{\rm(see \cite{L20142})}
		For any $p \in (0, \infty]$ and weight $w \in A_\infty$, denote by $L^p(\mathcal{X},w)$ the collection of all functions $f$ satisfying
		$$
		\|f\|_{L^p(\mathcal{X}, w)} := \left( \int_{\mathcal{X}} |f(y)|^p w(y) \, d\mu(y) \right)^{1/p} < \infty.
		$$
		Analogously, we denote by $L^{p, \infty}(\mathcal{X},w)$ the weak space with norm
		$$
		\|f\|_{L^{p, \infty}(\mathcal{X},w)} := \sup_{t>0} t \left( w\left( \{x \in \mathcal{X} : |f(x)| > t \} \right) \right)^{1/p},
		$$
		where $w(E) := \int_E w(x) \, d\mu(x)$ for all sets $E$ contained in $\mathcal{X}$.
		
	\end{definition}
	\begin{definition}\label{definition2.5}{\rm(see \cite{W2024})}
		Let $f\in L_{loc}^q(\mathcal{X})$ with $1\leq q<\infty $. For all balls $B \subset \mathcal{X}$, $f_B:=\frac{1}{\mu(B)}\int_{B}f(x)\d {\mu(x)}$. Then the space $BMO_q(\mathcal{X})$ is defined as follows. $$BMO_q(\mathcal{X}):=\{f\in L_{loc}^q(\mathcal{X}): \left\|f\right\|_{BMO_q(\mathcal{X})}<\infty\},$$
		where $$\left\|f\right\|_{BMO_q(\mathcal{X})}:=\left(\sup\limits_{B}\frac{1}{\mu(B)}\int_{B}|f(x)-f_B|^q\d {\mu(x)}\right)^{1/q}.$$
	\end{definition}
	
	\begin{remark}\label{Remark 2.2}
		\begin{itemize}
			\item[\rm (i)] If $q=1$, we simply write $BMO_{q}(\mathcal{X})$ as $BMO(\mathcal{X})$. For any $q\in (1,\infty)$, the norm of $BMO_q(\mathcal{X})$ is equivalent to $BMO(\mathcal{X})$.
			\item[\rm (ii)] For any $f\in L_{loc}^1(\mathcal{X})$, $$\left\|f\right\|_{BMO}\sim \sup\limits_{B}\inf\limits_{a\in \mathbb{C}}\frac{1}{\mu(B)}\int_B\left|f(x)-a\right|\d {\mu(x)} .$$
		\end{itemize}
	\end{remark}

	\begin{definition}\label{definition2.6}{\rm(see \cite{L20142})}
		For any $f\in L_{loc}^1(\mathcal{X})$ and $x\in \mathcal{X}$, the Hardy-Littlewood maximal function is defined as follows. $$M(f)(x):=\sup\limits_{B\ni x}\frac{1}{\mu(B)}\int_{B}|f(y)|\d{\mu(y)},$$
		where the supremum is taken over all balls $B$ in $\mathcal{X}$ containing $x$.
	\end{definition}
	
	\begin{lemma}\label{Remark 2.3}{\rm(see \cite{L20142})}
		$M$ is bounded on $L^{p,\infty}(\mathcal{X})$ for all $p\in (1,\infty)$.
	\end{lemma}
	
	\begin{definition}\label{definition2.7}{\rm(see \cite{L20142})}
		Let $\vec{f}:=(f_1,...,f_m)$ with every $f_j \in L_{loc}^1(\mathcal{X})$. The maximal operator $\mathcal{M}$ is defined by setting, for any $x\in \mathcal{X}$,
		$$\mathcal{M}(\vec{f})(x):=\sup\limits_{B\ni x}\prod\limits_{j=1}^m\frac{1}{\mu(B)}\int_{B}|f_j(y_j)|\d{\mu(y_j)},$$
		where the supremum is taken over all balls $B$ contained in $\mathcal{X}$ containing $x$.
	\end{definition}

	\begin{remark}\label{Remark 2.4}
		By the Definitions \ref{definition2.6} and \ref{definition2.7}, we can easily get that $$\mathcal{M}(\vec{f})(x) \leq \prod\limits_{j=1}^mM(f_j)(x).$$
	\end{remark}
	
	\begin{definition}\label{definition2.8}
		Let $f\in L_{loc}^1(\mathcal{X})$, the sharp  maximal function is defined as follows
		$$M^{\sharp}(f)(x):= \sup\limits_{B\ni x} \frac{1}{\mu(B)}\int_{B}|f(y)-f_B|\d {\mu(y)},$$
		where the supremum is taken over all balls $B$ contained in $\mathcal{X}$ containing $x$.
	\end{definition}
	
	In addition, for any $\delta \in (0,\infty )$, in this article we always use the following notation to simplify the representation of operators $$M_{\delta}(f)(x):=\left[M\left(\left|f\right|^{\delta}\right)(x)\right]^{1/\delta} \quad\text{and}\quad M_{\delta}^{\sharp}(f)(x):=\left[M^{\sharp}\left(\left|f\right|^{\delta}\right)(x)\right]^{1/\delta}.$$
	The same  for the maximal operator $\mathcal{M}$ with the following notation $$\mathcal{M}_{\delta}(\vec{f})(x):=\left[\mathcal{M}\left(|\vec{f} |^{\delta}\right)(x)\right]^{1/\delta}=\left(\sup\limits_{B\ni x}\prod\limits_{j=1}^m\frac{1}{\mu(B)}\int_{B}|f_j(y_j)|^{\delta}\d{\mu(y_j)}\right)^{1/\delta},$$
	 where $|\vec{f}|^{\delta}:=\left(|f_1|^{\delta},...,|f_m|^{\delta}\right)$.
	\begin{remark}{\rm(see \cite{W2024})}
		For a given $\delta \in (0,+\infty)$, it is easy to derive that $(a+b)^{\delta} \thicksim  a^{\delta} + b^{\delta}$ holds for any $a > 0$ and $b > 0$. Consequently, the following relation can be obtained:$$M_{\delta}^{\sharp}(f)(x) \thicksim  \sup\limits_{B\ni x}\inf\limits_{a\in \mathbb{C}}\left(\frac{1}{\mu(B)}\int_{B}\left||f(y)|^{\delta}-|a|^{\delta}\right|\d {\mu(y)}\right)^{1/\delta}.$$
		
	\end{remark}
	
	\begin{lemma}\label{Lemma2.1}{\rm(see \cite{W2024})}
		Let $(\mathcal{X},d,\mu)$ be a metric measure space and $0<p<r<\infty $.  There is a positive constant $C$ such that, for any measurable function $f$ and $B\subset \mathcal{X}$,$$\left[\mu(B)\right]^{-1/p}\left\|f\right\|_{L^p(B)}\leq C \left[\mu(B)\right]^{-1/r}\left\|f\right\|_{L^{r,\infty}(B)}.$$
	\end{lemma}
	
	\begin{lemma}\label{Lemma2.2}{\rm(see \cite{L20142})}
		Let $\vec{w}:=(w_1,...,w_m)$, $\nu_{\vec{w}}:=\prod\limits_{j=1}^mw_j^{p/p_j}$,  and $1\leq p_1,...,p_m<\infty$ with $1/p=1/p_1+\cdots+1/p_m$. Then $\vec{w} \in A_{\vec{p}}$ if and only if $\nu_{\vec{w}} \in A_{mp}$ and $w_j^{1-p_j'} \in A_{mp_j'}$ for all $j\in \{1,...,m\}$. When $p_j=1$, the condition $w_j^{1-p_j'} \in A_{mp_j'}$ is understood as $w_j^{1/m} \in A_1$.
	\end{lemma}
	
	\begin{lemma}\label{Lemma2.3}{\rm(see \cite{L20142})}
		Let $1 < p_1, ..., p_m < \infty $ with $1/p = 1/p_1+\cdots+ 1/p_m$. Then $\vec{w}:= (w_1,...,w_m) \in A_{\vec{p}}$ if and only if there exists a positive constant $C$ such that, for all $\vec{f}:=(f_1,...,f_m)$ with each $f_j \in L^{p_j}(\mathcal{X}, w_j)$ and $\nu_{\vec{w}}:=\prod\limits_{j=1}^mw_j^{p/p_j}$, one has, $$\left\|\mathcal{M}(\vec{f})\right\|_{L^p(\mathcal{X},\nu_{\vec{w}})}\leq C\prod\limits_{j=1}^m\left\|f_j\right\|_{L^{p_j}(\mathcal{X},w_j)}.$$ 
	\end{lemma}
	
	\begin{lemma}\label{Lemma2.4}{\rm(see \cite{L20142})}
		Let $1 \leq p_1, ..., p_m < \infty $ with $1/p = 1/p_1+\cdots+ 1/p_m$. Then $\vec{w}:= (w_1,...,w_m) \in A_{\vec{p}}$  if and only if there exists a positive constant $C$ such that,  for all $\vec{f}:=(f_1,...,f_m)$ with each $f_j \in L^{p_j}(\mathcal{X},w_j)$ and $\nu_{\vec{w}}:=\prod\limits_{j=1}^mw_j^{p/p_j}$, $$\left\|\mathcal{M}(\vec{f})\right\|_{L^{p,\infty}(\mathcal{X},\nu_{\vec{w}})}\leq C\prod\limits_{j=1}^m\left\|f_j\right\|_{L^{p_j}(\mathcal{X},w_j)}.$$ 
	\end{lemma}
	
		\begin{lemma}\label{Lemma2.5}{\rm(see \cite{L20142})}
		Let $0<p_0\leq  p<\infty$ and $w\in A_{\infty}$. Then, there exists a positive constant $C$ such that, for all $f\in L_{loc}^1(\mathcal{X})$ satisfying $M(f) \in L^{p_0,\infty}(w)$, we have \begin{itemize}
			\item [\rm(i)] If $p_0<p$, then $$\left\|M(f)\right\|_{L^{p}(\mathcal{X},w)} \leq C\left\|M^{\sharp}(f)\right\|_{L^{p}(\mathcal{X},w)}.$$
			\item[\rm (ii)] If $p_0 \leq p$, then $$\left\|M(f)\right\|_{L^{p,\infty}(\mathcal{X},w)} \leq C\left\|M^{\sharp}(f)\right\|_{L^{p,\infty}(\mathcal{X},w)}.$$
		\end{itemize}
	\end{lemma}
	
	\begin{lemma}\label{Lemma2.6}
		Let $0<r\leq p<\infty$ and  $w\in A_{\infty}$. For any $\delta \in (0,r)$, there exists a $C> 0 $ such that, for any $f \in L^{r,\infty}(\mathcal{X})$,
		\begin{itemize}
			\item[\rm (i)] If $p>r$, then  $$\left\|f\right\|_{L^p(\mathcal{X},w)} \leq C\left\|M_{\delta}^{\sharp}(f)\right\|_{L^p(\mathcal{X},w)}.$$
			\item[\rm (ii)] If $p\geq r$, then  $$\left\|f\right\|_{L^{p,\infty}(\mathcal{X},w)} \leq C\left\|M_{\delta}^{\sharp}(f)\right\|_{L^{p,\infty}(\mathcal{X},w)}.$$
		\end{itemize}
	
	\end{lemma}
	\begin{proof}
		For any $N\in \mathbb{N}$, set $w_N:=\min \{w,N\}$, then $w_N\in A_{\infty}$, by Fatou's lemma we have 
		\begin{align*}	
			\left\|f\right\|_{L^p(\mathcal{X},w)}^p& = \int_{\mathcal{X}}\lim\limits_{N\rightarrow \infty}|f(x)|^pw_N(x)\d {\mu(x)}\\
			 &\leq \liminf\limits_{N\rightarrow\infty}\int_{\mathcal{X}}|f(x)|^pw_N(x)\d {\mu(x)}\\
			 &= \liminf\limits_{N\rightarrow\infty}\left\|f\right\|_{L^p(\mathcal{X},w_N)}^p.
		\end{align*}
		 We know that $f \in L^{r,\infty}(\mathcal{X})$, and hence $\left|f\right|^{\delta} \in L_{loc}^1(\mathcal{X})$ for all $\delta \in (0, r)$. Then,  by Lemma \ref{Remark 2.3} we have
		\begin{align*}
			\left\|M(|f|^{\delta})\right\|_{L^{r/\delta,\infty}(\mathcal{X},w_N)}^{1/\delta}&= \left\|M_{\delta}\left(f\right)\right\|_{L^{r,\infty}(\mathcal{X},w_N)}\\
			&\lesssim N\left\|M_{\delta}\left(f\right)\right\|_{L^{r,\infty}(\mathcal{X})}\\
			&=N \left\|M\left(|f|^{\delta}\right)\right\|_{L^{r/\delta,\infty}(\mathcal{X})}^{1/\delta}\\
			&\lesssim N\left\||f|^{\delta}\right\|_{L^{r/\delta,\infty}(\mathcal{X})}^{1/\delta}\\
			&=N\left\|f\right\|_{L^{r,\infty}(\mathcal{X})}<\infty,
		\end{align*}
		which implies that $M(|f|^{\delta}) \in L^{r/\delta,\infty}(\mathcal{X},w_N)$. Then, using Lemma \ref{Lemma2.5} on $\left|f\right|^{\delta}$, we have $$\left\|f\right\|_{L^p(\mathcal{X},w_N)}\lesssim \left\|M(|f|^{\delta})\right\|_{L^{p/\delta}(\mathcal{X},w_N)}^{1/\delta}\lesssim \left\|M^{\sharp}(|f|^{\delta})\right\|_{L^{p/\delta}(\mathcal{X},w_N)}^{1/\delta}\lesssim \left\|M_{\delta}^{\sharp}\left(f\right)\right\|_{L^{p}(\mathcal{X},w)},$$
		which proves the (i) part of the lemma by letting $N\rightarrow \infty$. Using a similar method, we can establish the conclusion in (ii), and we omit the detais here for brevity.
	\end{proof}
	\begin{lemma}\label{Lemma2.7}
		Fix  $x\in \mathcal{X}$. Suppose $f\in BMO(\mathcal{X})$, $1\le p <\infty$ and $r_1,r_2>0$, then there is a constant $C>0$ independent of $f$, $x$, $r_1$ and $r_2$ such that $$\left(\frac{1}{\mu(B(x,r_1))}\int_{B(x,r_1)}\left|f(y)-f_{B(x,r_2)}\right|^p\d {\mu(y)}\right)^{1/p}\leq C\left(1+\left|\ln \frac{r_1}{r_2}\right|\right)\|f\|_{BMO(\mathcal{X})}.$$
	\end{lemma}
	\begin{proof}
		We only consider the case $0 < r_1 \leq r_2$. Actually, the similar procedure also works for another case $0 < r_2 < r_1$. For $0 < r_1 \leq r_2$, there are $k_1, k_2 \in \mathbb{Z}$ such that $2^{k_1 - 1} < r_1 \leq 2^{k_1}$ and $2^{k_2 - 1} < r_2 \leq 2^{k_2}$. Then $k_1 \leq k_2$ and
		$$(k_2 - k_1 - 1) \ln 2 < \ln \frac{r_2}{r_1} < (k_2 - k_1 + 1) \ln 2.$$
		Thus, we have
			\begin{align*}
				&\quad\left( \frac{1}{\mu(B(x, r_1))} \int_{B(x, r_1)} |f(y) - f_{B(x, r_2)}|^p \d {\mu(y)} \right)^{1/p}\\
				&\leq \left( \frac{1}{\mu(B(x, r_1))} \int_{B(x, r_1)} |f(y) - f_{B(x, 2^{k_1})}|^p \d {\mu(y)} \right)^{1/p} + |f_{B(x, 2^{k_1})} - f_{B(x, r_2)}|\\
		        &\leq \left( \frac{C_1}{\mu(B(x, 2^{k_1}))} \int_{B(x, 2^{k_1})} |f(y) - f_{B(x, 2^{k_1})}|^p \d {\mu(y)} \right)^{1/p} + |f_{B(x, r_2)} - f_{B(x, 2^{k_2})}|\\	
		        &\quad + \sum_{j = k_1}^{k_2 - 1} |f_{B(x, 2^{j+1})} - f_{B(x, 2^j)}|\\
		        &\leq C \|f\|_{\mathrm{BMO}} + \frac{1}{\mu(B(x, r_2))} \int_{B(x, r_2)} |f(y) - f_{B(x, 2^{k_2})}| \d {\mu(y)}\\
		        &\quad + \sum_{j = k_1}^{k_2 - 1} \frac{1}{\mu(B(x, 2^j))} \int_{B(x, 2^j)} |f(y) - f_{B(x, 2^{j+1})}| \d {\mu(y)}\\
		        &\leq C \|f\|_{\mathrm{BMO}} + \frac{C_1}{\mu(B(x, 2^{k_2}))} \int_{B(x, 2^{k_2})} |f(y) - f_{B(x, 2^{k_2})}| \d {\mu(y)}\\
		        &\quad + \sum_{j = k_1}^{k_2 - 1} \frac{C_1}{\mu(B(x, 2^{j+1}))} \int_{B(x, 2^{j+1})} |f(y) - f_{B(x, 2^{j+1})}| \d {\mu(y)}\\
			    &\leq \|f\|_{\mathrm{BMO}} \left( C + 2^n + 2^n (k_2 - k_1) \right)\\
			    &\leq C\left(1+\ln{\frac{r_2}{r_1}}\right)\|f\|_{BMO(\mathcal{X})}.
	\end{align*}
	\end{proof}

	\subsection{The Boundedness on Weighted Lebesgue Spaces}
	
	\qquad In this subsection, we initially obtain the sharp maximal estimates for $m$-linear strongly singular integral operators with generalized kernels and establish the boundedness on weighted Lebesgue spaces.
	
	\begin{theorem}\label{Theorem2.1}
		
		Suppose $T$ is an $m$-\text{\rm GSSIO} on \text{\rm RD}-space and $p_0'\geq\max\{r_1,...,r_m,l_1,...,l_m\}$, where $1/p_0+1/p_0'=1$, $p_0$, $r_j$ and $l_j$ are given as in Definition \ref{definition1.5}, $j=1,...,m.$ If  $0<\delta < 1/m$, then there is a constant $C>0$ such that for all $m$-tuples bounded measurable functions $\vec{f}:=(f_1,..., f_m)$  with bounded support,  $$M_{\delta}^{\sharp}\left(T(\vec{f})\right)(x) \leq C \mathcal{M}_{p_0'}(\vec{f})(x).$$
		
	\end{theorem}

	\begin{proof}
		We shall examine two cases for any ball $B :=B(z_B, r_B) $ containing $ x $, where $ r_B > 0 $.\\
			Case 1: $r_B \geq \displaystyle \frac{1}{4}$.\\
			
		Set $f_j^0:=f_j\chi_{32B}$ and $f_j^{\infty}:=f_j\chi_{(32B)^c}$. Then 
			\begin{align*}
				\prod\limits_{j=1}^mf_j(y_j)&=\prod\limits_{j=1}^mf_j^0(y_j) +\sum'f_1^{b_1}(y_1)\cdots f_m^{b_m}(y_m),
			\end{align*}
			where each term of $\sum'$, $b_j\in \{0,\infty\}$ and at least one $b_j = \infty$, $j=1,...,m$. Let $\vec{f}_0:=(f_1^0,... , f_m^0)$ and $\vec{f}_b:=(f_1^{b_1},..., f_m^{b_m})$. Then $$T(\vec{f})(z)= T(\vec{f_0})(z)+\sum'T(\vec{f_b})(z).$$
			
		Pick a $z_0\in 6B\backslash 5B$ and set $c_0:=\sum'T(\vec{f_b})(z_0)$. Then 
			\begin{align*}	
				&\quad\left(\frac{1}{\mu(B)}\int_B\left|\left|T(\vec{f})(z)\right|^{\delta}-|c_0|^{\delta}\right|\d {\mu(z)}\right)^{\frac{1}{\delta}}\\
				&\lesssim \left(\frac{1}{\mu(B)}\int_B\left|T(\vec{f})(z)-c_0\right|^{\delta}\d {\mu(z)}\right)^{\frac{1}{\delta}}\\
				&\lesssim \left(\frac{1}{\mu(B)}\int_B\left|T(\vec{f_0})(z)\right|^{\delta}\d {\mu(z)}\right)^{\frac{1}{\delta}}+\left(\frac{1}{\mu(B)}\int_B\left|\sum'T(\vec{f_b})(z)-\sum'T(\vec{f_b})(z_0)\right|^{\delta}\d {\mu(z)}\right)^{\frac{1}{\delta}}\\
				&:= I_1+I_2.
			\end{align*}
			
	Noting that $0<\delta<1/m \leq r$, by Lemma \ref{Lemma2.1}, (ii) of Definition \ref{definition1.5},  (\ref{1.3}), and H\"{o}lder's inequality, we have 
			\begin{align*}
				I_1&\lesssim \mu(B)^{-1/r}\left\|T(\vec{f_0})\right\|_{L^{r,\infty}(B)}\\
				&\lesssim \mu(B)^{-1/r}\prod\limits_{j=1}^m\left\|f_j^0\right\|_{L^{r_j}(\mathcal{X})}\\
				&\lesssim \mu(B)^{-1/r}\mu(32B)^{1/r}\prod\limits_{j=1}^m\left(\frac{1}{\mu(32B)}\int_{32B}|f_j(y_j)|^{r_j}\d {\mu(y_j)}\right)^{\frac{1}{r_j}}\\
				&\lesssim \prod\limits_{j=1}^m\left(\frac{1}{\mu(32B)}\int_{32B}|f_j(y_j)|^{p_0'}\d {\mu(y_j)}\right)^{\frac{1}{p_0'}}\\
				&\lesssim \mathcal{M}_{p_0'}(\vec{f})(x).
			\end{align*}
			
			For $I_2$, we have the following inequality
			\begin{align*}
				I_2&\lesssim \sum'\left(\frac{1}{\mu(B)}\int_B\left|T(\vec{f_b})(z)-T(\vec{f_b})(z_0)\right|^{\delta}\d {\mu(z)}\right)^{\frac{1}{\delta}}\\
				&\lesssim \sum'\left(\frac{1}{\mu(B)}\int_B\left|T(\vec{f_b})(z)-T(\vec{f_b})(z_0)\right|\d {\mu(z)}\right).
			\end{align*}
			
		Let $\vec{z}_0:=(z_0,...,z_0)$. For $z\in B$, $z_0\in 6B\backslash 5B$ and some $y_j\in (32B)^c$ with $j=1,...,m$, there are  $1\leq 4r_B \leq d(z,z_0) \leq 7r_B$ and $\rho(\vec{y},\vec{z}_0)=\max\limits_{1\leq i \leq m}d(y_i,z_0)\geq 2d(z,z_0)$. From H\"{o}lder's inequality,  (i) of  Definition \ref{definition1.5}, and (\ref{1.5}), it follows that,
			\begin{align*}
				&\quad\left|T(\vec{f_b})(z)-T(\vec{f_b})(z_0)\right|\\
				&\lesssim \int_{\mathcal{X}^m}\left|K(z,y_1,...,y_m)-K(z_0,y_1,...,y_m)\right|\prod\limits_{j=1}^m|f^{b_j}_j(y_j)|\d {\mathbf{\overrightarrow{\mu(y)}}}\\
				&\lesssim \sum\limits_{k=1}^{\infty}\left(\int_{2^kd(z,z_0)\leq \rho(\vec{y},\vec{z}_0)\le 2^{k+1}d(z,z_0)}|K(z,y_1,...,y_m)-K(z_0,y_1,...,y_m)|^{p_0}\d {\mathbf{\overrightarrow{\mu(y)}}}\right)^{\frac{1}{p_0}}\\
				&\quad\times \left(\int_{2^kd(z,z_0)\leq \rho(\vec{y},\vec{z}_0)\le 2^{k+1}d(z,z_0)}\prod\limits_{j=1}^m|f_j(y_j)|^{p_0'}\d {\mathbf{\overrightarrow{\mu(y)}}}\right)^{\frac{1}{p_0'}}\\
				&\lesssim \sum\limits_{k=1}^{\infty}d(z,z_0)^{\varepsilon-\left(\frac{n m}{p_0'}+\frac{\varepsilon}{\alpha}\right)}2^{-k\left(\frac{nm}{p_0'}+\frac{\varepsilon}{\alpha}\right)}\prod\limits_{j=1}^m\left(\int_{B(z_0,2^{k+1}d(z,z_0))}|f_j(y_j)|^{p_0'}\d {\mu(y_j)}\right)^{\frac{1}{p_0'}}\\
				&\lesssim \sum\limits_{k=1}^{\infty}d(z,z_0)^{\varepsilon-\left(\frac{n m}{p_0'}+\frac{\varepsilon}{\alpha}\right)}2^{-k\left(\frac{nm}{p_0'}+\frac{\varepsilon}{\alpha}\right)}\mu\left(B(z_0,2^{k+1}d(z,z_0))\right)^{\frac{m}{p_0'}}\mathcal{M}_{p_0'}(\vec{f})(x)\\
				&\lesssim  \sum\limits_{k=1}^{\infty}d(z,z_0)^{\varepsilon-\left(\frac{n m}{p_0'}+\frac{\varepsilon}{\alpha}\right)}2^{-k\left(\frac{nm}{p_0'}+\frac{\varepsilon}{\alpha}\right)}\left(2^{k+1}d(z,z_0)\right)^{\frac{nm}{p_0'}}\mathcal{M}_{p_0'}(\vec{f})(x)\\
				&\lesssim \sum\limits_{k=1}^{\infty}d(z,z_0)^{\varepsilon-\frac{\varepsilon}{\alpha}}2^{-\frac{k\varepsilon}{\alpha}}\mathcal{M}_{p_0'}(\vec{f})(x)\\
				&\lesssim r_B^{\varepsilon-\frac{\varepsilon}{\alpha}}\mathcal{M}_{p_0'}(\vec{f})(x)\\
				&\lesssim \mathcal{M}_{p_0'}(\vec{f})(x).
			\end{align*}
			
			Based on this result, we have the following estimate for $I_2$.
			\begin{align*}
				I_2&\lesssim \sum'\left(\frac{1}{\mu(B)}\int_B\left|T(\vec{f_b})(z)-T(\vec{f_b})(z_0)\right|\d {\mu(z)}\right)\\
				&\lesssim \mathcal{M}_{p_0'}(\vec{f})(x).
			\end{align*}\\
			Case 2 : $0< r_B < \displaystyle \frac{1}{4}$.\\
			
			 Set $ \tilde{B}: = B(z_B,r_B^{\alpha})$, $\tilde{f_j}^0 :=f_j\chi_{16\tilde{B}}$, and $ \tilde{f_j}^{\infty}:=f_j\chi_{(16\tilde{B})^c}$.  Then 
			\begin{align*}
				\prod\limits_{j=1}^mf_j(y_j)&=\prod\limits_{j=1}^m\tilde{f}_j^0(y_j) +\sum'\tilde{f}_1^{b_1}(y_1)\cdots \tilde{f}_m^{b_m}(y_m),
			\end{align*}
			where each term of $\sum'$, $b_j\in \{0,\infty\}$ and at least one $b_j = \infty,j=1,...,m$. Let $\vec{f}_{\tilde{0}}=(\tilde{f}_1^0,... , \tilde{f}_m^0)$ and $\vec{f}_{\tilde{b}}=(\tilde{f}_1^{b_1},... , \tilde{f}_m^{b_m})$. Then $$T(\vec{f})(z)= T(\vec{f}_{\tilde{0}})(z)+\sum'T(\vec{f}_{\tilde{b}})(z).$$
			
		Pick a $u_0\in 3B\backslash 2B$ and set $c_1:=\sum'T(\vec{f}_{\tilde{b}})(u_0)$. Then 
			\begin{align*}	
				&\quad\left(\frac{1}{\mu(B)}\int_B\left||T(\vec{f})(z)|^{\delta}-|c_1|^{\delta}\right|\d {\mu(z)}\right)^{\frac{1}{\delta}}\\
				&\lesssim \left(\frac{1}{\mu(B)}\int_B\left|T(\vec{f}_{\tilde{0}})(z)\right|^{\delta}\d {\mu(z)}\right)^{\frac{1}{\delta}}+\left(\frac{1}{\mu(B)}\int_B\left|\sum'T(\vec{f}_{\tilde{b}})(z)-\sum'T(\vec{f}_{\tilde{b}})(u_0)\right|^{\delta}\d {\mu(z)}\right)^{\frac{1}{\delta}}\\
				&:= \tilde{I}_1+\tilde{I}_2.
			\end{align*}
			
			 Note that $0< r_B < 1/4$, $\alpha\kappa/l - n/q \geq 0$, and $0<\delta<q<\infty$. By Lemma \ref{Lemma2.1}, (iii) of Definition \ref{definition1.5}, (\ref{1.3}), (\ref{1.6}) and H\"{o}lder's inequality, we have 
			\begin{align*}
				\tilde{I}_1&\lesssim \mu(B)^{-1/q}\left\|T(\vec{f}_{\tilde{0}})\right\|_{L^{q,\infty}(B)}\\
				&\lesssim \mu(B)^{-1/q}\prod\limits_{j=1}^m\left\|\tilde{f_j^0}\right\|_{L^{l_j}(\mathcal{X})}\\
				&=\mu(B)^{-1/q}\mu(16\tilde{B})^{1/l}\prod\limits_{j=1}^m\left(\frac{1}{\mu(16\tilde{B})}\int_{16\tilde{B}}|f_j(y_j)|^{l_j}\d {\mu(y_j)}\right)^{\frac{1}{l_j}}\\
				&\lesssim r_B^{\alpha\kappa/l-n/q}\prod\limits_{j=1}^m\left(\frac{1}{\mu(16\tilde{B})}\int_{16\tilde{B}}|f_j(y_j)|^{l_j}\d {\mu(y_j)}\right)^{\frac{1}{l_j}}\\
				&\lesssim r_B^{\alpha\kappa/l-n/q}\prod\limits_{j=1}^m\left(\frac{1}{\mu(16\tilde{B})}\int_{16\tilde{B}}|f_j(y_j)|^{p_0'}\d {\mu(y_j)}\right)^{\frac{1}{p_0'}}\\
				&\lesssim \mathcal{M}_{p_0'}(\vec{f})(x).
			\end{align*}
			
			For $\tilde{I}_2$, we have the following inequality
			\begin{align*}
				\tilde{I}_2&\lesssim \sum'\left(\frac{1}{\mu(B)}\int_B\left|T(\vec{f}_{\tilde{b}})(z)-T(\vec{f}_{\tilde{b}})(u_0)\right|^{\delta}\d {\mu(z)}\right)^{\frac{1}{\delta}}\\
				&\lesssim \sum'\left(\frac{1}{\mu(B)}\int_B\left|T(\vec{f}_{\tilde{b}})(z)-T(\vec{f}_{\tilde{b}})(u_0)\right|\d {\mu(z)}\right).
			\end{align*}
			
			Let $\vec{u}_0:=(u_0,...,u_0)$. For $z\in B$, $u_0\in 3B\backslash 2B$ and some $y_j\in (16\tilde{B})^c$ with   $j=1,...,m$, there are  $ r_B \leq d(z,u_0) \leq 4r_B < 1$ and $\rho(\vec{y},\vec{u}_0)=\max\limits_{1\leq i \leq m}d(y_i,u_0)\geq 2d(z,u_0)^{\alpha}$. From H\"{o}lder's inequality, (i) of Definition \ref{definition1.5}, (\ref{1.3})  and (\ref{1.6}),  we obtain that 
			\begin{align*}
				&\quad\left|T(\vec{f}_{\tilde{b}})(z)-T(\vec{f}_{\tilde{b}})(u_0)\right|\\
				&\lesssim \int_{\mathcal{X}^m}\left|K(z,y_1,...,y_m)-K(u_0,y_1,...,y_m)\right|\prod\limits_{j=1}^m\left|\tilde{f}^{b_j}_j(y_j)\right|\d {\mathbf{\overrightarrow{\mu(y)}}}\\
				&\lesssim \sum\limits_{k=1}^{\infty}\left(\int_{2^kd(z,u_0)^{\alpha}\leq \rho(\vec{y},\vec{u}_0)\leq 2^{k+1}d(z,u_0)^{\alpha}}\left|K(z,y_1,...,y_m)-K(u_0,y_1,...,y_m)\right|^{p_0}\d {\mathbf{\overrightarrow{\mu(y)}}}\right)^{\frac{1}{p_0}}\\
				&\quad\times \left(\int_{2^kd(z,u_0)^{\alpha}\leq \rho(\vec{y},\vec{u}_0)\le 2^{k+1}d(z,u_0)^{\alpha}}\prod\limits_{j=1}^m\left|f_j(y_j)\right|^{p_0'}\d {\mathbf{\overrightarrow{\mu(y)}}}\right)^{\frac{1}{p_0'}}\\
				&\lesssim \sum\limits_{k=1}^{\infty}d(z,u_0)^{\varepsilon-\alpha\left(\frac{\kappa m}{p_0'}+\frac{\varepsilon}{\alpha}\right)}2^{-k\left(\frac{nm}{p_0'}+\frac{\varepsilon}{\alpha}\right)}\prod\limits_{j=1}^m\left(\int_{B\left(u_0,2^{k+1}d(z,u_0)^{\alpha}\right)}|f_j(y_j)|^{p_0'}\d {\mu(y_j)}\right)^{\frac{1}{p_0'}}\\
				&\lesssim \sum\limits_{k=1}^{\infty}d(z,u_0)^{\varepsilon-\alpha\left(\frac{\kappa m}{p_0'}+\frac{\varepsilon}{\alpha}\right)}2^{-k\left(\frac{nm}{p_0'}+\frac{\varepsilon}{\alpha}\right)}\mu\left(B\left(u_0,2^{k+1}d(z,u_0)^{\alpha}\right)\right)^{\frac{m}{p_0'}}\mathcal{M}_{p_0'}(\vec{f})(x)\\
				&\lesssim  \sum\limits_{k=1}^{\infty}d(z,u_0)^{\varepsilon-\alpha\left(\frac{\kappa m}{p_0'}+\frac{\varepsilon}{\alpha}\right)}2^{-k\left(\frac{nm}{p_0'}+\frac{\varepsilon}{\alpha}\right)}2^{(k+1)\frac{nm}{p_0'}}d(z,u_0)^{\alpha\frac{\kappa m}{p_0'}}\mathcal{M}_{p_0'}(\vec{f})(x)\\
				&\lesssim \sum\limits_{k=1}^{\infty}2^{-\frac{k\varepsilon}{\alpha}}\mathcal{M}_{p_0'}(\vec{f})(x)\\
				&\lesssim \mathcal{M}_{p_0'}(\vec{f})(x).
			\end{align*}
			
			Based on this result, we have the following estimate for $\tilde{I}_2$,
			\begin{align*}
				\tilde{I_2}&\lesssim \sum'\left(\frac{1}{\mu(B)}\int_B\left|T(\vec{f}_{\tilde{b}})(z)-T(\vec{f}_{\tilde{b}})(u_0)\right|\d{\mu(z)}\right)\\
				&\lesssim \mathcal{M}_{p_0'}(\vec{f})(x).	
			\end{align*}
		Combining estimates from both cases, we conclude that $$M_{\delta}^{\sharp}\left(T(\vec{f})\right)(x)\sim \sup\limits_{B\ni x}\inf\limits_{a\in \mathbb{C}}\left(\frac{1}{\mu(B)}\int_B\left|\left|T(\vec{f})(z)\right|^{\delta}-|a|^{\delta}\right|\d {\mu(z)}\right)^{\frac{1}{\delta}} \leq C \mathcal{M}_{p_0'}(\vec{f})(x).$$
		This completes the proof.	
	\end{proof}
	
	\begin{theorem}\label{Theorem 2.2} 
		Suppose $T$ is an $m$-\text{\rm GSSIO} on \text{\rm RD}-space, $\displaystyle p_0'\geq\max\{r_1,...,r_m,l_1,...,l_m\}$ and $1/p_0 + 1/p_0' = 1,$ where $\displaystyle p_0,r_j$ and $l_j$ are given by Definition $\ref{definition1.5}$, $j=1,...,m$, $\vec{p}:=(p_1,...,p_m)$, $\vec{w}:=(w_1,...,w_m)\in A_{\vec{p}/p_0'}$, $1/p=1/p_1+\cdots+1/p_m$ and $\nu_{\vec{w}}:=\prod\limits_{j=1}^mw_j^{p/p_j}$. Then the following statements hold.
		\begin{itemize}
			\item [\rm (i)] If $p_0'<p_j<\infty$, $j=1,...,m$, then there exists a constant $C>0$ such that $$\left\|T(\vec{f})\right\|_{L^p(\mathcal{X},\nu_{\vec{w}})} \leq C \prod\limits_{j=1}^m\left\|f_j\right\|_{L^{p_j}(\mathcal{X},w_j)}.$$ 
			
			\item [\rm (ii)]  If $p_0'\leq p_j<\infty$, $j=1,...,m$, and at least one of $p_j =p_0'$, then there exists a constant $C>0$ such that $$\left\|T(\vec{f})\right\|_{L^{p,\infty}(\mathcal{X},\nu_{\vec{w}})} \leq C \prod\limits_{j=1}^m\left\|f_j\right\|_{L^{p_j}(\mathcal{X},w_j)}.$$
		\end{itemize}

	\end{theorem}
	
	\begin{proof}
		By Lemma \ref{Lemma2.2}, we can get that $\nu_{\vec{w}} \in A_{mp/p_0'} \subset A_{\infty}$. Taking a $\delta$ such that $0<\delta < 1/m$, then  applying  Lemma \ref{Lemma2.6}, Theorem \ref{Theorem2.1}, and Lemma \ref{Lemma2.3}, for the case (i), we have 
		\begin{align*}
			\left\|T(\vec{f})\right\|_{L^p(\mathcal{X},\nu_{\vec{w}})}&\lesssim  \left\|M_{\delta}^{\sharp}\left(T(\vec{f})\right)\right\|_{L^p(\mathcal{X},\nu_{\vec{w}})} \lesssim \left\|\mathcal{M}_{p_0'}(\vec{f})\right\|_{L^p(\mathcal{X},\nu_{\vec{w}})}\\
			&= \left\|\mathcal{M}\left(|\vec{f}|^{p_0'}\right)\right\|_{L^{p/p_0'}(\mathcal{X},\nu_{\vec{w}})}^{1/p_0'} \lesssim \prod\limits_{j=1}^m\left\||f_j|^{p_0'}\right\|_{L^{p_j/p_0'}(\mathcal{X},w_j)}^{1/p_0'}\\
			&=\prod\limits_{j=1}^m\left\|f_j\right\|_{L^{p_j}(\mathcal{X},w_j)}.
		\end{align*}
		The proof of (i) is completed. By Lemma \ref{Lemma2.6}, Theorem \ref{Theorem2.1}, and Lemma \ref{Lemma2.4}, for the case (ii), we have 
			\begin{align*}
			\left\|T(\vec{f})\right\|_{L^{p,\infty}(\mathcal{X},\nu_{\vec{w}})}&\lesssim  \left\|M_{\delta}^{\sharp}\left(T(\vec{f})\right)\right\|_{L^{p,\infty}(\mathcal{X},\nu_{\vec{w}})} \lesssim \left\|\mathcal{M}_{p_0'}(\vec{f})\right\|_{L^{p,\infty}(\mathcal{X},\nu_{\vec{w}})}\\
			&=\left\|\mathcal{M}\left(|\vec{f}|^{p_0'}\right)\right\|_{L^{p/p_0',\infty}(\mathcal{X},\nu_{\vec{w}})}^{1/p_0'} \lesssim \prod\limits_{j=1}^m\left\||f_j|^{p_0'}\right\|_{L^{p_j/p_0'}(\mathcal{X},w_j)}^{1/p_0'}\\
			&=\prod\limits_{j=1}^m\left\|f_j\right\|_{L^{p_j}(\mathcal{X},w_j)}.
		\end{align*}
		The proof of (ii) is completed.
	\end{proof}
	\begin{remark}
		\begin{itemize}
			\item[\rm (i)]When $\mathcal{X} = \mathbb{R}^n$, i.e., the RD space reduces to the classical Euclidean space, our boundedness results for the operator on weighted Lebesgue spaces coincide with Theorem 2.2 established by Wei et al. in \cite{B2023}. Additionally,  the weak-type result is new even for the $\mathbb{R}^n$ spaces.
			
			\item[\rm (ii)]If the condition of the kernel is strengthened to the pointwise estimate, and the parameters $\kappa$ and $n$ of the \text{\rm RD}-space  satisfy $\kappa = n$, the boundedness results we obtained on weighted Lebesgue spaces recover Theorem 1.1 established by Li and Wu in \cite{W2024}. Furthermore, we establish a weak-type result on weighted Lebesgue spaces, which is new in this setting, and our framework is more general.
		\end{itemize}
	\end{remark}
	\section{Endpoint Estimations}\label{sec3}
	
	\qquad Firstly, we recall some essential definitions and introduce the notations that will be used subsequently. Then
	we establish the boundedness of $m$-linear strongly singular integral operator with generalized kernel from $L^{\infty}(\mathcal{X})\times \cdots \times L^{\infty}(\mathcal{X})$ into $BMO(\mathcal{X})$
	and from $BMO(\mathcal{X})\times \cdots \times BMO(\mathcal{X})$ into $BMO(\mathcal{X})$.
	
	\subsection{Boundedness of $L^{\infty}(\mathcal{X}) \times \cdot \cdot \cdot \times L^{\infty}(\mathcal{X}) \rightarrow BMO(\mathcal{X})$ Type}

	\begin{theorem}\label{Theorem 3.1} 
		Suppose $T$ is an $m$-\text{\rm GSSIO} on \text{\rm RD}-space, $\displaystyle p_0,q,r_j,l_j$ are given by Definition $\ref{definition1.5}$, $j=1,... ,m$ and $q>1$, $ p_0'\geq\max\{r_1,...,r_m,l_1,...,l_m\}$ satisfied $1/p_0 + 1/p_0' = 1.$  Then, $T$ can be extended to a bounded operator from $L^{\infty}(\mathcal{X})\times \cdots \times L^{\infty}(\mathcal{X})$ into $BMO(\mathcal{X})$.
	\end{theorem}
	
	\begin{proof}
		We shall examine two cases for any ball \( B: = B(z_B, r_B) \subset \mathcal{X} \), where \( r_B > 0 \).\\
	Case 1: $r_B \geq \displaystyle \frac{1}{4}$.\\
	
			 Using the same decomposition and notation as in case 1 of the proof of Theorem \ref{Theorem2.1}, we have
			\begin{align*}	
				\frac{1}{\mu(B)}\int_B\left|T(\vec{f})(z)-c_0\right|\d {\mu(z)}&
				\lesssim \frac{1}{\mu(B)}\int_B\left|T(\vec{f_0})(z)\right|\d {\mu(z)}\\
				&\quad+\frac{1}{\mu(B)}\int_B\left|\sum'T(\vec{f_b})(z)-\sum'T(\vec{f_b})(z_0)\right|\d {\mu(z)}\\
				&:= J_1+J_2.
			\end{align*}
			
			Pick $p_1,...,p_m$ such that $\max\{m,p_0'\}< p_1,...,p_m<\infty$ and $1/p=1/p_1+\cdots+1/p_m$. From H\"{o}lder's inequality, Theorem \ref{Theorem 2.2}, and   (\ref{1.3}), we have 
			\begin{align*}
				J_1&\lesssim\left(\frac{1}{\mu(B)}\int_B\left|T(\vec{f_0})(z)\right|^p\d {\mu(z)}\right)^{1/p}\\
				&\lesssim \mu(B)^{-1/p}\prod\limits_{j=1}^m\left\|f_j^0\right\|_{L^{p_j}(\mathcal{X})}\\
				&\lesssim \mu(B)^{-1/p}\mu(32B)^{1/p}\prod\limits_{j=1}^m\left\|f_j\right\|_{L^{\infty}(\mathcal{X})}\\
				&\lesssim \prod\limits_{j=1}^m\left\|f_j\right\|_{L^{\infty}(\mathcal{X})}.
			\end{align*}
			
		For $J_2$, we have the following inequality
			\begin{align*}
				J_2&\lesssim \sum'\left(\frac{1}{\mu(B)}\int_B\left|T(\vec{f_b})(z)-T(\vec{f_b})(z_0)\right|\d {\mu(z)}\right).
			\end{align*}
			
		By  H\"{o}lder's inequality, (i) of  Definition \ref{definition1.5}, and (\ref{1.5}), we obtain 
			\begin{align*}
				&\qquad\left|T(\vec{f_b})(z)-T(\vec{f_b})(z_0)\right|\\
				&\lesssim \int_{\mathcal{X}^m}\left|K(z,y_1,...,y_m)-K(z_0,y_1,...,y_m)\right|\prod\limits_{j=1}^m|f^{b_j}_j(y_j)|\d {\mathbf{\overrightarrow{\mu(y)}}}\\
				&\lesssim \sum\limits_{k=1}^{\infty}\left(\int_{2^kd(z,z_0)\leq \rho(\vec{y},\vec{z}_0)\le 2^{k+1}d(z,z_0)}|K(z,y_1,...,y_m)-K(z_0,y_1,...,y_m)|^{p_0}\d {\mathbf{\overrightarrow{\mu(y)}}}\right)^{\frac{1}{p_0}}\\
				&\quad\times \left(\int_{2^kd(z,z_0)\leq \rho(\vec{y},\vec{z}_0)\le 2^{k+1}d(z,z_0)}\prod\limits_{j=1}^m|f_j(y_j)|^{p_0'}\d {\mathbf{\overrightarrow{\mu(y)}}}\right)^{\frac{1}{p_0'}}\\
				&\lesssim \sum\limits_{k=1}^{\infty}d(z,z_0)^{\varepsilon-(\frac{n m}{p_0'}+\frac{\varepsilon}{\alpha})}2^{-k(\frac{nm}{p_0'}+\frac{\varepsilon}{\alpha})}\prod\limits_{j=1}^m\left(\int_{B\left(z_0,2^{k+1}d(z,z_0)\right)}|f_j(y_j)|^{p_0'}\d {\mu(y_j)}\right)^{\frac{1}{p_0'}}\\
				&\lesssim \sum\limits_{k=1}^{\infty}d(z,z_0)^{\varepsilon-\left(\frac{n m}{p_0'}+\frac{\varepsilon}{\alpha}\right)}2^{-k\left(\frac{nm}{p_0'}+\frac{\varepsilon}{\alpha}\right)}\mu\left(B(z_0,2^{k+1}d(z,z_0))\right)^{\frac{m}{p_0'}}\prod\limits_{j=1}^m\left\|f_j\right\|_{L^{\infty}(\mathcal{X})}\\
				&\lesssim \sum\limits_{k=1}^{\infty}d(z,z_0)^{\varepsilon-\frac{\varepsilon}{\alpha}}2^{-\frac{k\varepsilon}{\alpha}}\prod\limits_{j=1}^m\left\|f_j\right\|_{L^{\infty}(\mathcal{X})}\\
				&\lesssim r_B^{\varepsilon-\frac{\varepsilon}{\alpha}}\prod\limits_{j=1}^m\left\|f_j\right\|_{L^{\infty}(\mathcal{X})}\\
				&\lesssim \prod\limits_{j=1}^m\left\|f_j\right\|_{L^{\infty}(\mathcal{X})}.
			\end{align*}
			
			Based on this result, we have the following estimate for $J_2$,
			\begin{align*}
				J_2\lesssim \prod\limits_{j=1}^m\left\|f_j\right\|_{L^{\infty}(\mathcal{X})}.
			\end{align*}
			\\
			Case 2:  $0< r_B <\displaystyle \frac{1}{4}$.\\ 
			
			Using the same notations as in case 2 of the proof of Theorem \ref{Theorem2.1}, we have
			\begin{align*}	
				\frac{1}{\mu(B)}\int_B\left|T(\vec{f}_{\tilde{0}})(z)-c_1\right|\d {\mu(z)}
				&\lesssim \frac{1}{\mu(B)}\int_B\left|T(\vec{f}_{\tilde{0}})(z)\right|\d {\mu(z)}\\
				&\quad+\frac{1}{\mu(B)}\int_B\left|\sum'T(\vec{f}_{\tilde{b}})(z)-\sum'T(\vec{f}_{\tilde{b}})(u_0)\right|\d {\mu(z)}\\
				&:= \tilde{J}_1+\tilde{J}_2.
			\end{align*}
			
			Note that $1<q< \infty$ and $0<r_B <1/4$, $\alpha\kappa/l - n/q \geq 0$. From Lemma \ref{Lemma2.1}, (iii) of Definition \ref{definition1.5}, (\ref{1.3}), (\ref{1.6}), and H\"{o}lder's inequality, it follows that 
			\begin{align*}
				\tilde{J}_1&\lesssim \mu(B)^{-1/q}\left\|T(\vec{f}_{\tilde{0}})\right\|_{L^{q,\infty}(B)}\\
				&\lesssim \mu(B)^{-1/q}\prod\limits_{j=1}^m\left\|\tilde{f}_j^0\right\|_{L^{l_j}(\mathcal{X})}\\
				&\lesssim \mu(B)^{-1/q}\mu(16\tilde{B})^{1/l}\prod\limits_{j=1}^m\left\|f_j\right\|_{L^{\infty}(\mathcal{X})}\\
				&\lesssim r_B^{\alpha\kappa/l-n/q}\prod\limits_{j=1}^m\left\|f_j\right\|_{L^{\infty}(\mathcal{X})}\\
				&\lesssim \prod\limits_{j=1}^m\left\|f_j\right\|_{L^{\infty}(\mathcal{X})}.
			\end{align*}
			
		For $\tilde{J}_2$, we have the following inequality
			\begin{align*}
				\tilde{J_2}\lesssim \sum'\left(\frac{1}{\mu(B)}\int_B\left|T(\vec{f}_{\tilde{b}})(z)-T(\vec{f}_{\tilde{b}})(u_0)\right|\d {\mu(z)}\right).
			\end{align*}
			
			As  the proof of Theorem \ref{Theorem2.1}, we have 	
			\begin{align*}
				&\quad\left|T(\vec{f}_{\tilde{b}})(z)-T(\vec{f}_{\tilde{b}})(u_0)\right|\\
				&\lesssim \int_{\mathcal{X}^m}\left|K(z,y_1,...,y_m)-K(u_0,y_1,...,y_m)\right|\prod\limits_{j=1}^m\left|\tilde{f}^{b_j}_j(y_j)\right|\d {\mathbf{\overrightarrow{\mu(y)}}}\\
				&\lesssim \sum\limits_{k=1}^{\infty}\left(\int_{2^kd(z,u_0)^{\alpha}\leq \rho(\vec{y},\vec{u}_0)\leq 2^{k+1}d(z,u_0)^{\alpha}}\left|K(z,y_1,...,y_m)-K(u_0,y_1,...,y_m)\right|^{p_0}\d {\mathbf{\overrightarrow{\mu(y)}}}\right)^{\frac{1}{p_0}}\\
				&\quad\times \left(\int_{2^kd(z,u_0)^{\alpha}\leq \rho(\vec{y},\vec{u}_0)\le 2^{k+1}d(z,u_0)^{\alpha}}\prod\limits_{j=1}^m\left|f_j(y_j)\right|^{p_0'}\d {\mathbf{\overrightarrow{\mu(y)}}}\right)^{\frac{1}{p_0'}}\\
				&\lesssim \sum\limits_{k=1}^{\infty}d(z,u_0)^{\varepsilon-\alpha\left(\frac{\kappa m}{p_0'}+\frac{\varepsilon}{\alpha}\right)}2^{-k\left(\frac{nm}{p_0'}+\frac{\varepsilon}{\alpha}\right)}\mu\left(B(z_0,2^{k+1}d(z,u_0)^{\alpha})\right)^{\frac{m}{p_0'}}\prod\limits_{j=1}^m\left\|f_j\right\|_{L^{\infty}(\mathcal{X})}\\
				&\lesssim \sum\limits_{k=1}^{\infty}2^{-\frac{k\varepsilon}{\alpha}}\prod\limits_{j=1}^m\left\|f_j\right\|_{L^{\infty}(\mathcal{X})}\\
				&\lesssim \prod\limits_{j=1}^m\left\|f_j\right\|_{L^{\infty}(\mathcal{X})}.
			\end{align*}
			
			Based on this result, we have the following estimate for $\tilde{J}_2$,
			\begin{align*}
				\tilde{J}_2\lesssim \prod\limits_{j=1}^m\left\|f_j\right\|_{L^{\infty}(\mathcal{X})}.
			\end{align*}
		Combining estimates from both cases, there is $$\left\|T(\vec{f})\right\|_{BMO}\sim \sup\limits_{B}\inf\limits_{a\in \mathbb{C}}\frac{1}{\mu(B)}\int_B\left|T(\vec{f})(z)-a\right|\d {\mu(z)} \leq C \prod\limits_{j=1}^m\left\|f_j\right\|_{L^{\infty}(\mathcal{X})},$$
		which completes the proof.	
	\end{proof}
	
	\subsection{Boundedness of $BMO(\mathcal{X}) \times \cdot \cdot \cdot \times BMO(\mathcal{X}) \rightarrow BMO(\mathcal{X})$ Type}
	
	\begin{theorem}\label{Theorem 3.2} 
		Suppose $T$ is an $m$-\text{\rm GSSIO} on  \text{\rm RD}-space, $\displaystyle p_0,q,r_j,l_j$ are given by Definition $\ref{definition1.5}$, $j=1,...,m$ and $q>1$, $ p_0'\geq\max\{r_1,...,r_m,l_1,...,l_m\}$, $T(f_1,...,f_{j-1},1,f_{j+1},...,f_m)=0(j=1,...,m)$.  Then, $T$ can be extended to a bounded operator from $BMO(\mathcal{X})\times \cdots \times BMO(\mathcal{X})$ into $BMO(\mathcal{X})$.
	\end{theorem}	
	\begin{proof}
		To simplify the proof, we only provide the case when $m = 2$, since the case for $m>2$ is similar. Let $f_1,f_2\in BMO(\mathcal{X})$, then for any ball $B :=B(z_B,r_B)$ with $r_B > 0$, we  consider two cases.\\
			Case 1: $r_B \geq \frac{1}{4}$.\\
			
			 Write
			\begin{align*}
				&f_1=(f_1)_{32B}+ (f_1-(f_1)_{32B})\chi_{32B} + (f_1-(f_1)_{32B})\chi_{(32B)^c}:=f_1^1+f_1^2+f_1^3\\
				&f_2=(f_2)_{32B}+ (f_2-(f_2)_{32B})\chi_{32B} + (f_2-(f_2)_{32B})\chi_{(32B)^c}:=f_2^1+f_2^2+f_2^3.
			\end{align*}
			
		Noting that $T(f_1,...,f_{j-1},1,f_{j+1},...,f_m)=0$, we can get the following result$$T(f_1,f_2)=T(f_1^2,f_2^2)+T(f_1^2,f_2^3)+T(f_1^3,f_2^2)+T(f_1^3,f_2^3).$$
		
			Pick a $z_0\in 6B\backslash 5B$ and set $c_2:=T(f_1^2,f_2^3)(z_0)+T(f_1^3,f_2^2)(z_0)+T(f_1^3,f_2^3)(z_0)$. Then 
			\begin{align*}
				\frac{1}{\mu(B)}\int_{B}\left|T(f_1,f_2)(z)-c_2\right|\d {\mu(z)}&\lesssim \frac{1}{\mu(B)}\int_{B}\left|T(f_1^2,f_2^2)(z)\right|\d {\mu(z)}\\
				&\quad+\frac{1}{\mu(B)}\int_{B}\left|T(f_1^2,f_2^3)(z)-T(f_1^2,f_2^3)(z_0)\right|\d {\mu(z)}\\
				&\quad+\frac{1}{\mu(B)}\int_{B}\left|T(f_1^3,f_2^2)(z)-T(f_1^3,f_2^2)(z_0)\right|\d {\mu(z)}\\
				&\quad+\frac{1}{\mu(B)}\int_{B}\left|T(f_1^3,f_2^3)(z)-T(f_1^3,f_2^3)(z_0)\right|\d {\mu(z)}\\
				&:=I+II+III+IV.
			\end{align*}
			
		According to Theorem \ref{Theorem 2.2}, we  choose $p_1, p_2 >\max\{p_0', 2\}$, and $1/p=1/p_1+1/p_2$, such that $T$ is bounded from $L^{p_1}(\mathcal{X})\times L^{p_2}(\mathcal{X})$ to $L^p(\mathcal{X})$, then by applying H\"{o}lder's inequality and (\ref{1.3}), we obtain 
			\begin{align*}
				I&\lesssim \left(\frac{1}{\mu(B)}\int_{B}\left|T(f_1^2,f_2^2)(z)\right|^p\d {\mu(z)}\right)^{\frac{1}{p}}\\
				&\lesssim \mu(B)^{-1/p}\left\|f_1^2\right\|_{L^{p_1}(\mathcal{X})}\left\|f_2^2\right\|_{L^{p_2}(\mathcal{X})}\\
				&\lesssim \mu(B)^{-1/p}\mu(32B)^{1/p}\prod\limits_{j=1}^2\left(\frac{1}{\mu(32B)}\int_{32B}|f_j(y_j)-(f_j)_{32B}|^{p_j}\d {\mu(z)}\right)^{\frac{1}{p_j}}\\
				&\lesssim \left\|f_1\right\|_{BMO(\mathcal{X})}\left\|f_2\right\|_{BMO(\mathcal{X})}.
			\end{align*}
			
		Let $\vec{z}_0:=(z_0,z_0)$. For $z\in B$, $z_0\in 6B\backslash 5B$ and  $y_2\in (32B)^c$, there are  $1\leq 4r_B \leq d(z,z_0) \leq 7r_B$ and $\rho(\vec{y},\vec{z}_0)\geq d(y_2,z_0)\geq 2d(z,z_0)$. By H\"{o}lder's inequality, (i) of  Definition \ref{definition1.5}, Lemma \ref{Lemma2.7}, (\ref{1.3})  and (\ref{1.5}),  we have,
			\begin{align*}
				&\quad\left|T(f_1^2,f_2^3)(z)-T(f_1^2,f_2^3)(z_0)\right|\\
				&\lesssim \int_{\mathcal{X}^2}|K(z,y_1,y_2)-K(z_0,y_1,y_2)||f_1^2(y_1)||f_2^3(y_2)|\d {\mu(y_1)}\d {\mu(y_2)}\\
				&\lesssim \sum\limits_{k=1}^{\infty}\left(\int_{2^kd(z,z_0)\leq \rho(\vec{y},\vec{z}_0)\le 2^{k+1}d(z,z_0)}|K(z,y_1,y_2)-K(z_0,y_1,y_2)|^{p_0}\d {\mu(y_1)}\d {\mu(y_2)}\right)^{\frac{1}{p_0}}\\
				&\quad\times \left(\int_{2^kd(z,z_0)\leq \rho(\vec{y},\vec{z}_0)\le 2^{k+1}d(z,z_0)}|f_1^2(y_1)f_2^3(y_2)|^{p_0'}\d {\mu(y_1)}\d {\mu(y_2)}\right)^{\frac{1}{p_0'}}\\
				&\lesssim \sum\limits_{k=1}^{\infty}d(z,z_0)^{\varepsilon-\left(\frac{2n }{p_0'}+\frac{\varepsilon}{\alpha}\right)}2^{-k\left(\frac{2n}{p_0'}+\frac{\varepsilon}{\alpha}\right)}\prod\limits_{j=1}^2\left(\int_{B\left(z_0,2^{k+1}d(z,z_0)\right)}|f_j(y_j)-(f_j)_{32B}|^{p_0'}\d {\mu(y_j)}\right)^{\frac{1}{p_0'}}\\
				&\lesssim \sum\limits_{k=1}^{\infty}d(z,z_0)^{\varepsilon-\left(\frac{2n }{p_0'}+\frac{\varepsilon}{\alpha}\right)}2^{-k\left(\frac{2n}{p_0'}+\frac{\varepsilon}{\alpha}\right)}\prod\limits_{j=1}^2\left(\int_{2^{k+5}B}|f_j(y_j)-(f_j)_{32B}|^{p_0'}\d {\mu(y_j)}\right)^{\frac{1}{p_0'}}\\
				&\lesssim \sum\limits_{k=1}^{\infty}d(z,z_0)^{\varepsilon-\left(\frac{2n }{p_0'}+\frac{\varepsilon}{\alpha}\right)}2^{-k\left(\frac{2n}{p_0'}+\frac{\varepsilon}{\alpha}\right)}\left(1+\left|\ln\frac{2^{k+5}r_B}{32r_B}\right|\right)^2\mu\left(2^{k+5}B\right)^{\frac{2}{p_0'}}\prod\limits_{j=1}^2\left\|f_j\right\|_{BMO(\mathcal{X})}\\
				&\lesssim  \sum\limits_{k=1}^{\infty}r_B^{\varepsilon-\frac{\varepsilon}{\alpha}}2^{-\frac{k\varepsilon}{\alpha}}k^2\prod\limits_{j=1}^2\left\|f_j\right\|_{BMO(\mathcal{X})}\\
				&\lesssim \sum\limits_{k=1}^{\infty}k^22^{-\frac{k\varepsilon}{\alpha}}\prod\limits_{j=1}^2\left\|f_j\right\|_{BMO(\mathcal{X})}\\
				&\lesssim \left\|f_1\right\|_{BMO(\mathcal{X})}\left\|f_2\right\|_{BMO(\mathcal{X})}.
			\end{align*}
			
			 Then we can get that $$II\lesssim \left\|f_1\right\|_{BMO(\mathcal{X})}\left\|f_2\right\|_{BMO(\mathcal{X})}.$$
			The estimates for $III$ and $IV$ are similar to the case of $II$, and we  omit the detalis.\\
			Case 2: $0< r_B < \frac{1}{4}$.\\
			
			 Set $ \tilde{B}: = B(z_B,r_B^{\alpha})$, 
			\begin{align*}
				&f_1=(f_1)_{16\tilde{B}}+ (f_1-(f_1)_{16\tilde{B}})\chi_{16\tilde{B}} + (f_1-(f_1)_{16\tilde{B}})\chi_{(16\tilde{B})^c}:=\tilde{f}_1^1+\tilde{f}_1^2+\tilde{f}_1^3\\
				&f_2=(f_2)_{16\tilde{B}}+ (f_2-(f_2)_{16\tilde{B}})\chi_{16\tilde{B}} + (f_2-(f_2)_{16\tilde{B}})\chi_{(16\tilde{B})^c}:=\tilde{f}_2^1+\tilde{f}_2^2+\tilde{f}_2^3.
			\end{align*}
			
			Noting that $T(f_1,...,f_{j-1},1,f_{j+1},...,f_m)=0$, we can get the following result$$T(f_1,f_2)=T(\tilde{f}_1^2,\tilde{f}_2^2)+T(\tilde{f}_1^2,\tilde{f}_2^3)+T(\tilde{f}_1^3,\tilde{f}_2^2)+T(\tilde{f}_1^3,\tilde{f}_2^3).$$
			
			Pick a $u_0\in 3B\backslash 2B$ and set $c_3:=T(\tilde{f}_1^2,\tilde{f}_2^3)(u_0)+T(\tilde{f}_1^3,\tilde{f}_2^2)(u_0)+T(\tilde{f}_1^3,\tilde{f}_2^3)(u_0)$. Then 
			\begin{align*}
				\frac{1}{\mu(B)}\int_{B}\left|T(f_1,f_2)(z)-c_3\right|\d {\mu(z)}&\lesssim \frac{1}{\mu(B)}\int_{B}\left|T(\tilde{f}_1^2,\tilde{f}_2^2)(z)\right|\d {\mu(z)}\\
				&\quad+\frac{1}{\mu(B)}\int_{B}\left|T(\tilde{f}_1^2,\tilde{f}_2^3)(z)-T(\tilde{f}_1^2,\tilde{f}_2^3)(u_0)\right|\d {\mu(z)}\\
				&\quad+\frac{1}{\mu(B)}\int_{B}\left|T(\tilde{f}_1^3,\tilde{f}_2^2)(z)-T(\tilde{f}_1^3,\tilde{f}_2^2)(u_0)\right|\d {\mu(z)}\\
				&\quad+\frac{1}{\mu(B)}\int_{B}\left|T(\tilde{f}_1^3,\tilde{f}_2^3)(z)-T(\tilde{f}_1^3,\tilde{f}_2^3)(u_0)\right|\d {\mu(z)}\\
				&:=\tilde{I}+\tilde{II}+\tilde{III}+\tilde{IV}.
			\end{align*}
			
			By Lemma \ref{Lemma2.1},  (iii) of Definition \ref{definition1.5}, (\ref{1.3}), (\ref{1.6}), and H\"{o}lder's inequality, we have 
			\begin{align*}
				\tilde{I}&\lesssim \mu(B)^{-1/q}\left\|T(\tilde{f}_1^2,\tilde{f}_2^2)\right\|_{L^{q,\infty}(B)}\\
				&\lesssim \mu(B)^{-1/q}\left\|\tilde{f}_1^2\right\|_{L^{l_1}(\mathcal{X})}\left\|\tilde{f}_2^2\right\|_{L^{l_2}(\mathcal{X})}\\
				&=\mu(B)^{-1/q}\mu(16\tilde{B})^{1/l}\prod\limits_{j=1}^2\left(\frac{1}{\mu(16\tilde{B})}\int_{16\tilde{B}}|f_j(y_j)-(f_j)_{16\tilde{B}}|^{l_j}\d {\mu(y_j)}\right)^{\frac{1}{l_j}}\\
				&\lesssim r_B^{\alpha\kappa/l-n/q}\left\|f_1\right\|_{BMO_{l_1}(\mathcal{X})}\left\|f_2\right\|_{BMO_{l_2}(\mathcal{X})}\\
				&\lesssim \left\|f_1\right\|_{BMO(\mathcal{X})}\left\|f_2\right\|_{BMO(\mathcal{X})}.	
			\end{align*}
			
			 Let $\vec{u}_0:=(u_0,u_0)$. For $z\in B, u_0\in 3B\backslash 2B$ and  $y_2\in (16\tilde{B})^c$, there are  $r_B \leq d(z,u_0) \leq 4r_B< 1$ and $\rho(\vec{y},\vec{u}_0)\geq d(y_2,u_0)\geq 2d(z,u_0)^{\alpha}$. From H\"{o}lder's inequality, (i) of Definition \ref{definition1.5}, Lemma \ref{Lemma2.7}, (\ref{1.3})  and (\ref{1.6}), we have 
			\begin{align*}
				&\quad\left|T(\tilde{f}_1^2,\tilde{f}_2^3)(z)-T(\tilde{f}_1^2,\tilde{f}_2^3)(u_0)\right|\\
				&\lesssim \int_{\mathcal{X}^2}|K(z,y_1,y_2)-K(u_0,y_1,y_2)||\tilde{f}_1^2(y_1)||\tilde{f}_2^3(y_2)|\d {\mu(y_1)}\d {\mu(y_2)}\\
				&\lesssim \sum\limits_{k=1}^{\infty}\left(\int_{2^kd(z,u_0)^{\alpha}\leq \rho(\vec{y},\vec{u}_0)\le 2^{k+1}d(z,u_0)^{\alpha}}|K(z,y_1,y_2)-K(u_0,y_1,y_2)|^{p_0}\d {\mu(y_1)}\d {\mu(y_2)}\right)^{\frac{1}{p_0}}\\
				&\quad\times \left(\int_{2^kd(z,u_0)^{\alpha}\leq \rho(\vec{y},\vec{u}_0)\leq 2^{k+1}d(z,u_0)^{\alpha}}|\tilde{f}_1^2(y_1)\tilde{f}_2^3(y_2)|^{p_0'}\d {\mu(y_1)}\d {\mu(y_2)}\right)^{\frac{1}{p_0'}}\\
				&\lesssim \sum\limits_{k=1}^{\infty}d(z,u_0)^{\varepsilon-\alpha\left(\frac{2\kappa }{p_0'}+\frac{\varepsilon}{\alpha}\right)}2^{-k\left(\frac{2n}{p_0'}+\frac{\varepsilon}{\alpha}\right)}\prod\limits_{j=1}^2\left(\int_{B\left(u_0,2^{k+1}d(z,u_0)^{\alpha}\right)}|f_j(y_j)-(f_j)_{16\tilde{B}}|^{p_0'}\d {\mu(y_j)}\right)^{\frac{1}{p_0'}}\\
				&\lesssim \sum\limits_{k=1}^{\infty}d(z,u_0)^{\varepsilon-\alpha\left(\frac{2\kappa }{p_0'}+\frac{\varepsilon}{\alpha}\right)}2^{-k\left(\frac{2n}{p_0'}+\frac{\varepsilon}{\alpha}\right)}\prod\limits_{j=1}^2\left(\int_{2^{k+5}\tilde{B}}|f_j(y_j)-(f_j)_{16\tilde{B}}|^{p_0'}\d {\mu(y_j)}\right)^{\frac{1}{p_0'}}\\
				&\lesssim \sum\limits_{k=1}^{\infty}d(z,u_0)^{\varepsilon-\alpha\left(\frac{2\kappa }{p_0'}+\frac{\varepsilon}{\alpha}\right)}2^{-k\left(\frac{2n}{p_0'}+\frac{\varepsilon}{\alpha}\right)}\left(1+\left|\ln\frac{2^{k+5}r_B^{\alpha}}{16r_B^{\alpha}}\right|\right)^2\mu\left(2^{k+5}\tilde{B}\right)^{\frac{2}{p_0'}}\prod\limits_{j=1}^2\left\|f_j\right\|_{BMO(\mathcal{X})}\\
				&\lesssim \sum\limits_{k=1}^{\infty}k^22^{-\frac{k\varepsilon}{\alpha}}\prod\limits_{j=1}^2\left\|f_j\right\|_{BMO(\mathcal{X})}\\
				&\lesssim \left\|f_1\right\|_{BMO(\mathcal{X})}\left\|f_2\right\|_{BMO(\mathcal{X})}.
			\end{align*}
			
		Then we can get that $$\tilde{II}\lesssim \left\|f_1\right\|_{BMO(\mathcal{X})}\left\|f_2\right\|_{BMO(\mathcal{X})}.$$
		
			The estimates for $\tilde{III}$ and $\tilde{IV}$ are similar to the case of $\tilde{II}$, and we also omit their details.
		Combining estimates from both cases, we have $$\left\|T(f_1,f_2)\right\|_{BMO(\mathcal{X})}\sim \sup\limits_{B}\inf\limits_{a\in \mathbb{C}}\frac{1}{\mu(B)}\int_B\left|T(f_1,f_2)(z)-a\right|\d {\mu(z)} \leq C\left\|f_1\right\|_{BMO(\mathcal{X})}\left\|f_2\right\|_{BMO(\mathcal{X})}.$$
		This completes the proof of the theorem.
		
	\end{proof}
	
	\section{Boundedness on Generalized Weighted Morrey Spaces}\label{sec4}

	\qquad First, we recall some essential definitions  and introduce the notations that will be used later. Then, we establish the boundedness of the $m$-linear strongly singular integral operator with  generalized kernel on generalized  weighted Morrey spaces.

	\begin{definition}\label{definition4.1}{\rm(see \cite{J2020})}
		Let $p\in [1,\infty)$, $w$ be a   weight, and $\varphi$ be a positive and increasing function on $\mathbb{R}^+$. Then  the generalized weight  Morrey 	space $L^{p,\varphi} (\mathcal{X},w)$   is defined by setting$$L^{p,\varphi}(\mathcal{X},w):=\{f\in L_{loc}^p(\mathcal{X}) : \left\|f\right\|_{L^{p,\varphi}(\mathcal{X},w)}<\infty\},$$
		where $$\left\|f\right\|_{L^{p,\varphi}(\mathcal{X},w)}:= \sup\limits_{B}\left(\frac{1}{\varphi(w(B))}\int_{B}|f(x)|^pw(x)\d {\mu(x)}\right)^{1/p},$$
		and the supremum is taken over all balls $B$ in $\mathcal{X}$.
	\end{definition}
	
	\begin{remark}
		\begin{itemize}
			\item[\rm (i)] If $\mathcal{X}= \mathbb{R}^n$, the generalized weight  Morrey 	space is introduced	by Komori and  Shiral \cite{Y2009}. When $w\equiv 1$, $L^{p,\varphi}(\mathbb{R}^n)$ is just the generalized Morrey space in \cite{Y2009}.
			\item[\rm (ii)] When $\varphi(t)= t^k $ with $0<k<1$, then $L^{p,\varphi}(\mathcal{X},w)= L^{p,k}(\mathcal{X},w)$ is  the weight Morrey space on \text{\rm RD}-spaces \cite{V2017}.  When $\varphi \equiv 1$, then $L^{p,\varphi}(\mathcal{X},w)=L^p(\mathcal{X},w)$.
		\end{itemize}

	\end{remark}
	
	\begin{lemma}\label{Lemma4.1}
			Let $\varphi$ be a positive and increasing function, $1\leq r< p<\infty$ and  $w\in A_{\infty}$. For any $\delta \in (0,r)$, there exists a  $C> 0 $ such that, for any $f \in L^{r,\infty}(\mathcal{X})$,
			$$\left\|f\right\|_{L^{p,\varphi}(\mathcal{X},w)} \leq C\left\|M_{\delta}^{\sharp}(f)\right\|_{L^{p,\varphi}(\mathcal{X},w)}.$$
			
	\end{lemma}
	
	\begin{proof}
		For any $B\subset \mathcal{X}$, note that $w\chi_B(y)\in  A_{\infty}$. From Lemma \ref{Lemma2.6}, we have 
		\begin{align*}
			&\quad\left(\frac{1}{\varphi(w(B))}\int_{B}|f(y)|^pw(y)\d {\mu(y)}\right)^{1/p}\\
			&=\left(\frac{1}{\varphi(w(B))}\int_{\mathcal{X}}|f(y)|^pw\chi_B(y)\d {\mu(y)}\right)^{1/p}\\
			&\leq \left(\frac{C}{\varphi(w(B))}\int_{\mathcal{X}}|M_{\delta}^{\sharp}(f)(y)|^pw\chi_B(y)\d {\mu(y)}\right)^{1/p}\\
			&=\left(\frac{C}{\varphi(w(B))}\int_{B}|M_{\delta}^{\sharp}(f)(y)|^pw(y)\d {\mu(y)}\right)^{1/p}\\
			&\leq C\left\|M_{\delta}^{\sharp}(f)\right\|_{L^{p,\varphi}(\mathcal{X},w)}.
		\end{align*}
		By taking the supremum of $B \subset \mathcal{X}$, the proof will be completed. 
	\end{proof}
	
	\begin{lemma}\label{Lemma4.2}{\rm(see \cite{L2014})}
		Let $w\in A_{\infty}$. Then for all balls $B\subset \mathcal{X}$, the following reverse Jensen inequality holds. $$\int_{B}w(x)\d {\mu(x)}\leq C\mu(B)\cdot \displaystyle  e^{\displaystyle \left(\frac{1}{\mu(B)}\int_{B}\log{w(x)}\d {\mu(x)}\right)}.$$
	\end{lemma}
	
	\begin{lemma}\label{Lemma4.3}
		Let $p_1,...,p_m\in [1,\infty)$ with $1/p=1/p_1+\cdots+1/p_m$. Assume that $w_1,...,w_m\in A_{\infty}$ and $\nu_{\vec{w}}:=\prod\limits_{j=1}^mw_j^{p/p_j}$, then there exists a constant $C>0$ such that for any ball $B$, $$\prod\limits_{j=1}^m\left(\int_{B}w_j(x)\d {\mu(x)}\right)^{p/p_j}\leq C \int_{B}\nu_{\vec{w}}(x)\d {\mu(x)}.$$
	\end{lemma}
	
	\begin{proof}
		Since $w_j\in A_{\infty}$, by Lemma \ref{Lemma4.2} and Jensen inequality (see \cite{L2014}), we have
		\begin{align*} 
			&\quad\prod\limits_{j=1}^m\left(\int_{B}w_j(x)\d {\mu(x)}\right)^{p/p_j}\\
			&\lesssim \prod\limits_{j=1}^m\left(\mu(B)\cdot e^{\displaystyle\left(\frac{1}{\mu(B)}\int_{B}\log{w_j(x)}\d {\mu(x)}\right)}\right)^{p/p_j}\\
			&\lesssim \mu(B)\prod\limits_{j=1}^m \displaystyle  e^{\displaystyle\left(\frac{1}{\mu(B)}\int_{B}\log{w_j(x)}^{p/p_j}\d {\mu(x)}\right)}\\
			&\lesssim \mu(B)e^{\displaystyle\left(\frac{1}{\mu(B)}\int_{B}\log\nu_{\vec{w}}(x)\d {\mu(x)}\right)}\\
			&\lesssim \int_{B}\nu_{\vec{w}}(x)\d {\mu(x)}.
		\end{align*}
	\end{proof}

	\begin{lemma}\label{Lemma4.4}
		Let $\varphi$, $\varphi_j$ be  positive and increasing functions on $\mathbb{R}^+$, and for all $t, t_j, k\in \mathbb{R}^{+}$, $j=1,...,m$,
		\begin{equation}\label{a}
			\prod\limits_{j=1}^m\varphi_j(t_j)\lesssim \varphi(\prod\limits_{j=1}^mt_j),\qquad \varphi(tk)\lesssim \varphi(t)\varphi(k),	\qquad \varphi_j(t)^k \lesssim \varphi_j(t^k).
		\end{equation}
		If $1\leq p_j <\infty$, $1/p=1/p_1+\cdots+1/p_m$, $p\geq 1$, $w_j\in A_{\infty}$  and $\nu_{\vec{w}}:=\prod\limits_{j=1}^mw_j^{p/p_j}$, $f_j\in L^{p_j,\varphi_j}(\mathcal{X},w_j)$, $j=1,...,m$, then $$\left\|\prod\limits_{j=1}^mf_j\right\|_{L^{p,\varphi}(\mathcal{X},\nu_{\vec{w}})}\leq C\prod\limits_{j=1}^m\left\|f_j\right\|_{L^{p_j,\varphi_j}(\mathcal{X},w_j)}.$$ 
	\end{lemma}
	
	\begin{proof}   Note that $\varphi$ be a positive and increasing function on $\mathbb{R}^+$, by Lemma \ref{Lemma4.3} and (\ref{a}), we have
		$$\varphi\left(\prod\limits_{j=1}^mw_j(B)^{p/p_j}\right)\leq \varphi\left(C\nu_{\vec{w}}(B)\right)\lesssim \varphi\left(\nu_{\vec{w}}(B)\right).$$
		Then, by H\"{o}lder's inequality and (\ref{a}), we have 
		\begin{align*}
			\left\|\prod\limits_{j=1}^mf_j\right\|_{L^{p,\varphi}(\nu_{\vec{w}})}&=\sup\limits_{B}\left(\frac{1}{\varphi(\nu_{\vec{w}}(B))}\int_{B}\left|\prod\limits_{j=1}^mf_j(x)\right|^p\nu_{\vec{w}}(x)\d {\mu(x)}\right)^{1/p}\\
			&\lesssim\sup\limits_{B}\frac{1}{\prod\limits_{j=1}^m\varphi_j\left(w_j(B)^{p/p_j}\right)^{1/p}}\prod\limits_{j=1}^m\left(\int_{B}|f_j(x)|^{p_j}w_j(x)\d {\mu(x)}\right)^{1/p_j}\\
			&\lesssim\sup\limits_{B}\frac{1}{\prod\limits_{j=1}^m\varphi_j\left(w_j(B)\right)^{1/p_j}}\prod\limits_{j=1}^m\left(\int_{B}|f_j(x)|^{p_j}w_j(x)\d {\mu(x)}\right)^{1/p_j}\\
			&\lesssim\prod\limits_{j=1}^m\sup\limits_{B}\left(\frac{1}{\varphi_j(w_j(B))}\int_{B}|f_j(x)|^{p_j}w_j(x)\d {\mu(x)}\right)^{1/p_j}\\
			&=C\prod\limits_{j=1}^m\left\|f_j\right\|_{L^{p_j,\varphi_j}(w_j)},
		\end{align*}
		where $C$ is independent of $f_j$.
	\end{proof}

	\begin{lemma}\label{Lemma4.5}{\rm(see \cite{J2020})}
		Let $p>1$, $w\in A_p$ and  $\varphi$ be a positive and increasing function on $\mathbb{R}^+$. Suppose there exists a constant $c$ such that for any $s\geq t>0$,  \begin{equation}\label{b}
			\frac{\varphi(s)}{s}\leq c \frac{\varphi(t)}{t}.
		\end{equation}
		Then  there exists a positive constant $C$ such that, for any $f\in L^{p,\varphi}(\mathcal{X},w)$,
		$$\left\|Mf\right\|_{L^{p,\varphi}(\mathcal{X},w)} \leq C\left\|f\right\|_{L^{p,\varphi}(\mathcal{X},w)}.$$ 
	\end{lemma}
	
	\begin{lemma}\label{Lemma4}{\rm(see \cite{L2004})}
		For $(w_1,...,w_m)\in (A_{p_1},...,A_{p_m})$ with $1\leq p_1,...,p_m<\infty $, and for any $0<\theta_1,...,\theta_m<1$ such that $\theta_1+\cdots+\theta_m=1$, we have $w_1^{\theta_1}\cdots w_m^{\theta_m}\in A_{\max\{p_1,...,p_m\}}.$
	\end{lemma}
	
	\begin{theorem}\label{Theorem4.1}
		Suppose $T$ is an $m$-\text{\rm GSSIO} on  \text{\rm RD}-space, $\displaystyle p_0,q,r_j,l_j$ are given by Definition $\ref{definition1.5}$, $j=1,... ,m$ and $ p_0'\geq\max\{r_1,...,r_m,l_1,...,l_m\}, 1/p_0 + 1/p_0' = 1.$ Let $\varphi$, $\varphi_j$ be  positive and increasing functions on $\mathbb{R}^+$, such that (\ref{a}) and (\ref{b}) for $\varphi_j$ hold, $w_j\in A_{p_j/p_0'}$, $1/p=1/p_1+\cdots+1/p_m$ and $\nu_{\vec{w}}:=\prod\limits_{j=1}^mw_j^{p/p_j}$. If $p_0'<p_j<\infty$ and $p\geq 1$, $j=1,...,m$, then there exists a constant $C>0$ such that $$\left\|T(\vec{f})\right\|_{L^{p,\varphi}(\mathcal{X},\nu_{\vec{w}})} \leq C \prod\limits_{j=1}^m\left\|f_j\right\|_{L^{p_j,\varphi_j}(\mathcal{X},w_j)}.$$ 
	\end{theorem}
	
	\begin{proof}
By Lemma \ref{Lemma4}, we can get that $\nu_{\vec{w}} \in A_{\max\{A_{p_1/p_0'},...,A_{p_m/p_0'}\}} \subset A_{\infty}$. Taking a $\delta$ such that $0<\delta < 1/m$,  by Lemma \ref{Lemma4.1}, Theorem \ref{Theorem2.1}, Remark \ref{Remark 2.4},  Lemma \ref{Lemma4.4}, and Lemma \ref{Lemma4.5}, we have 
\begin{align*}
	\left\|T(\vec{f})\right\|_{L^{p,\varphi}(\mathcal{X},\nu_{\vec{w}})} & \lesssim \left\|M_{\delta}^{\sharp}(T(\vec{f}))\right\|_{L^{p,\varphi}(\mathcal{X},\nu_{\vec{w}})}\lesssim \left\|\mathcal{M}_{p_0'}(\vec{f})\right\|_{L^{p,\varphi}(\mathcal{X},\nu_{\vec{w}})}\\
	&\lesssim \left\|\prod\limits_{j=1}^mM_{p_0'}(f_j)\right\|_{L^{p,\varphi}(\mathcal{X},\nu_{\vec{w}})}\lesssim \prod\limits_{j=1}^m\left\|M_{p_0'}(f_j)\right\|_{L^{p_j,\varphi_j}(\mathcal{X},w_j)}\\
	&=\prod\limits_{j=1}^m\left\|M(|f_j|^{p_0'})\right\|_{L^{p_j/p_0',\varphi_j}(\mathcal{X},w_j)}^{1/p_0'}\lesssim \prod\limits_{j=1}^m\left\||f_j|^{p_0'}\right\|_{L^{p_j/p_0',\varphi_j}(\mathcal{X},w_j)}^{1/p_0'}\\
	&=\prod\limits_{j=1}^m\left\|f_j\right\|_{L^{p_j,\varphi_j}(\mathcal{X},w_j)}.
\end{align*}
This completes the proof.
	\end{proof}
When we return to the weighted Morrey space on RD-space, that is, by taking $\varphi(t)=\varphi_1(t)=...=\varphi_m(t)=t^k$, $0<k<1$, both conditions (\ref{a})  and (\ref{b}) are naturally satisfied. Then we can obtain the boundedness on the weighted Morrey space, which is the corollary as follows.
	\begin{corollary}
		Suppose $T$ is an $m$-\text{\rm GSSIO} on  \text{\rm RD}-space, $\displaystyle p_0,q,r_j,l_j$ are given by Definition $\ref{definition1.5}$, $j=1,... ,m$ and $ p_0'\geq\max\{r_1,...,r_m,l_1,...,l_m\}, 1/p_0 + 1/p_0' = 1.$ Let $w_j\in A_{p_j/p_0'}$, $\nu_{\vec{w}}:=\prod\limits_{j=1}^mw_j^{p/p_j}$, and $1/p=1/p_1+\cdots+1/p_m$. If $0<k<1$, $p_0'<p_j<\infty$ and $p\geq 1$, $j=1,...,m$, then there exists a constant $C>0$ such that $$\left\|T(\vec{f})\right\|_{L^{p,k}(\mathcal{X},\nu_{\vec{w}})} \leq C \prod\limits_{j=1}^m\left\|f_j\right\|_{L^{p_j,k}(\mathcal{X},w_j)}.$$ 
	\end{corollary}
	\begin{remark}
	 When considering weighted Morrey spaces on \text{\rm RD}-space, our results are also novel. Even when we return to the Euclidean space, the Morrey space on \text{\rm RD}-space corresponds to the Morrey space in Euclidean space, and our results remain the most recent.
	\end{remark}

	\bigskip

	\noindent Kang Chen and Yan Lin
	
	\smallskip
	
	\noindent School of Science, China University of Mining and
	Technology, Beijing 100083,  People's Republic of China
	
	\smallskip
	
	\noindent{\it E-mails:} \texttt{chenkang@student.cumtb.edu.cn} (K. Chen)
	
	\noindent\phantom{{\it E-mails:} }\texttt{linyan@cumtb.edu.cn} (Y. Lin)

	\bigskip
	
	\noindent Shuhui Yang 
	
	\smallskip
	
	\noindent  Laboratory of Mathematics and Complex
	Systems (Ministry of Education of China),
	School of Mathematical Sciences, Beijing Normal
	University, Beijing 100875, People's Republic of China
	
	\smallskip
	
	\noindent{\it E-mail:} \texttt{shuhuiyang@bnu.edu.cn}


\begin{thebibliography}{99}
		
			\bibitem[1]{F1983} F. J. Almgren, Q-valued functions minimizing Dirichlet’s integral and the regularity of area minimizing rectifiable currents up to codimension two, Bull. Amer. Math. Soc. (N.S.) 8 (1983), 327--328.
		
		\vspace{-.3cm}
		
		\bibitem[2]{J1986} J. Alvarez and M. Milman, $H^p$ continuous properties of Calder\'{o}n--Zygmund-type operators, J. Math. Anal. Appl. 118 (1986), 63--79.
		
		\vspace{-.3cm}
		
		\bibitem[3]{L1977} L. A. Caffarelli, The regularity of free boundaries in higher dimensions, Acta Math. 139 (1977), 155--184.
		
		\vspace{-.3cm}
		
		\bibitem[4]{L2007}L. A. Caffarelli and L. Silvestre, An extension problem related to the fractional Laplacian, Comm. Partial Differential Equations 32 (2007), 1245--1260.
		
		\vspace{-.3cm}
		
		\bibitem[5]{A1952}A. P. Calder\'{o}n and A. Zygmund, On the existence of certain singular integrals, Acta Math. 88 (1952), 85--139.
		
		\vspace{-.3cm}
		
		\bibitem[6]{A1954} A. P. Calder\'{o}n and A. Zygmund, Singular integrals and periodic functions, Studia Math. 14 (1954), 249--271.
		
		\vspace{-.3cm}
		
		\bibitem[7]{A1957}A. P. Calder\'{o}n and A. Zygmund, Singular integral operators and differential equations, Amer. J. Math. 79 (1957), 901--921.
		
		\vspace{-.3cm}
		
		\bibitem[8]{J2020}J. Chou, X. Li, Y. Tong, and H. Lin, Generalized weighted Morrey spaces on RD-spaces, Rocky Mountain J. Math. 50 (2020), no. 4, 1277--1293.
		
		\vspace{-.3cm}
		
		\bibitem[9]{M1990} M. Christ, A $T(b)$ theorem with remarks on analytic capacity and the Cauchy integral, Colloq. Math. 60/61 (1990), 601--628.
		
		\vspace{-.3cm}
		
		\bibitem[10]{M1987} M. Christ and J. L. Journ\'{e}, Polynomial growth estimates for multilinear singular integral operators, Acta Math. 159 (1987), 51--80.
		
		\vspace{-.3cm}
		
		\bibitem[11]{R1982} R. R. Coifman, A. McIntosh, and Y. Meyer, L’int\'{e}grale de Cauchy d\'{e}finit un op\'{e}rateur born\'{e} sur $L^2$ pour les courbes lipschitziennes, Ann. of Math. (2) 116 (1982), 361--387.
		
		\vspace{-.3cm}
		
		\bibitem[12]{R1971} R. R. Coifman and G. Weiss, Analyse Harmonique Non-commutative sur Certains Espaces Homog\`{e}nes, Lecture Notes in Math., vol. 242, Springer, Berlin, 1971.
		
		\vspace{-.3cm}
		
		\bibitem[13]{R1977} R. R. Coifman and G. Weiss, Extensions of Hardy spaces and their use in analysis, Bull. Amer. Math. Soc. 83 (1977), 569--645.
		
		\vspace{-.3cm}
		
		\bibitem[14]{R1975} R. R. Coifman and Y. Meyer, On commutators of singular integrals and bilinear singular integrals, Trans. Amer. Math. Soc. 212 (1975), 315--331.
		
		\vspace{-.3cm}
		
		\bibitem[15]{R1978} R. R. Coifman and Y. Meyer, Au del\`{a} des op\'{e}rateurs pseudo-diff\'{e}rentiels(French), Ast\'{e}risque, 57, 1978.
		
		\vspace{-.3cm}
		
		\bibitem[16]{R19782}R. R. Coifman and Y. Meyer, Commutateurs d’int\'{e}grales singulières et op\'{e}rateurs multilinéaires, Ann. Inst. Fourier (Grenoble) 28 (1978), 177--202.
		
		\vspace{-.3cm}
		
		\bibitem[17]{G1991} G. David and S. Semmes, Singular integrals and rectifiable sets in $\mathbb{R}^n$: Beyond Lipschitz graphs, Astérisque, No. 193 (1991), 152 pp.
		
		\vspace{-.3cm}
		
		\bibitem[18]{C1970} C. Fefferman, Inequalities for strongly singular convolution operators, Acta Math. 124 (1970), 9--36.
		
		\vspace{-.3cm}
		
		\bibitem[19]{C1972} C. Fefferman and E. M. Stein, $H^p$ spaces of several variables, Acta Math. 129 (1972), 137--193.
		
		\vspace{-.3cm}
		
		\bibitem[20]{L2014} L. Grafakos, Classical Fourier Analysis, 3rd ed., Springer, New York, 2014.
		
		\vspace{-.3cm}
		
		\bibitem[21]{L2004} L. Grafakos and J. M. Martell, Extrapolation of weighted norm inequalities for multivariable operators and applications, J. Geom. Anal. 14 (2004), 19--46.
		
		\vspace{-.3cm}
		
		\bibitem[22]{L20142}L. Grafakos, L. Liu, D. Maldonado, and D. Yang, Multilinear analysis on metric spaces, Dissertationes Math. 497 (2014).
		
		
		\vspace{-.3cm}
		
		\bibitem[23]{L2002}L. Grafakos and R. H. Torres, Multilinear Calder\'{o}n--Zygmund theory, Adv. Math. 165 (2002), 124--164.
		
		\vspace{-.3cm}
		
		\bibitem[24]{Y2006} Y. Han, D. M\"{u}ller, and D. Yang, Littlewood-Paley characterizations for Hardy spaces on spaces of homogeneous type, Math. Nachr. 279 (2006), 1505--1537.
		
		\vspace{-.3cm}
		
		\bibitem[25]{Y2008} Y. Han, D. M\"{u}ller, and D. Yang, A theory of Besov and Triebel-Lizorkin spaces on metric measure spaces modeled on Carnot-Carath\'{e}odory spaces, Abstr. Appl. Anal. 2008, Art. ID 893409, 250 pp.
		
		\vspace{-.3cm}
		
		\bibitem[26]{I1959} I. I. Hirschman, Multiplier transformations, Duke Math. J. 26 (1959), 222--242.
		
		\vspace{-.3cm}
		
		\bibitem[27]{V2017}V. Kokilashvili and A. Meskhi, The boundedness of sublinear operators in weighted Morrey spaces defined on spaces of homogeneous type, in \textit{Function Spaces and Inequalities}, Springer Proc. Math. Stat. 206, Springer, 2017, 193--211.
		
		\vspace{-.3cm}
		
		\bibitem[28]{Y2009}Y. Komori and S. Shirai, Weighted Morrey spaces and an integral operator, Math. Nachr. 282 (2009), 219--231.
		
		\vspace{-.3cm}
		
		\bibitem[29]{P2010}P. Koskela, D. Yang, and Y. Zhou, A characterization of Haj lasz-Sobolev and Triebel-Lizorkin spaces via grand Littlewood-Paley functions, J. Funct. Anal. 258 (2010), 2637--2661.
		
		\vspace{-.3cm}
		
		\bibitem[30]{P2011} P. Koskela, D. Yang, and Y. Zhou, Pointwise characterizations of Besov and Triebel-Lizorkin spaces and quasiconformal mappings, Adv. Math. 226 (2011), 3579--3621.
		
		\vspace{-.3cm}
		
		\bibitem[31]{J2005} J. Li, The estimates for sharp maximal functions of multilinear strongly singular integral operators, Acta Math. Sin. (Engl. Ser.) {21} (2005),  1495--1508.
		
		\vspace{-.3cm}
		
		\bibitem[32]{J20062} J. Li and S. Lu, $L^p$ estimates for multilinear operators of strongly singular integral operators, Nagoya Math. J. {181} (2006), 41--62; MR2210709
		
		\vspace{-.3cm}
		
		\bibitem[33]{J2006} J. Li and S. Lu, Strongly singular integral operators on weighted Hardy space, Acta Math. Sin. (Engl. Ser.) {22} (2006),  767--772.
		
		\vspace{-.3cm}
		
		\bibitem[34]{W2024} W. Li and L. Wu, Weighted norm inequalities for multilinear strongly singular Calder\'{o}n--Zygmund operators on RD-spaces, Math. Nachr. 297 (2024), 657--680.
		
		\vspace{-.3cm}
		
		\bibitem[35]{Y2020} Y. Lin, Multilinear theory of strongly singular Calder\'{o}n--Zygmund operators and applications, Nonlinear Anal. 192 (2020), 111699.
		
		\vspace{-.3cm}
		
		\bibitem[36]{Y20192} Y. Lin and Y. Han, Sharp maximal and weighted estimates for multilinear iterated commutators of multilinear strongly singular Calder\'{o}n--Zygmund operators, Chinese J. Contemp. Math. 40 (2019), 399--416.
		
		\vspace{-.3cm}
		
		\bibitem[37]{Y2019} Y. Lin, G.  Lu, and S. Lu, Sharp maximal estimates for multilinear commutators of multilinear strongly singular Calder\'{o}n--Zygmund operators and applications, Forum Math. 31 (2019), 1--18.
		
		\vspace{-.3cm}
		
		\bibitem[38]{E1997} E. Nakai and K. Yabuta, Pointwise multipliers for functions of weighted bounded mean oscillation on spaces of homogeneous type, Math. Japon. 46 (1997), 15--28.
		
		\vspace{-.3cm}
		
		\bibitem[39]{F2003} F. Nazarov, S. Treil, and A. Volberg, The $T(b)$-theorem on non-homogeneous spaces, Acta Math. 190 (2003), 151--239.
		
		\vspace{-.3cm}
		
		\bibitem[40]{E1967}E. M. Stein, Singular integrals, harmonic functions and differentiability properties of functions of several variables, Proc. Sympos. Pure Appl. Math. 10 (1967), 316--335.
		
		
		\vspace{-.3cm}
		
		\bibitem[41]{E1970}E. M. Stein, Singular Integrals and Differentiability Properties of Functions, Princeton Univ. Press, Princeton, NJ, 1970.
		
		
		\vspace{-.3cm}
		
		\bibitem[42]{J1989}J. O. Str\"{o}mberg and A. Torchinsky, Weighted Hardy spaces, Lecture Notes in Mathematics, vol. 1381, Springer-Verlag, Berlin, 1989.
		
		\vspace{-.3cm}
		
		\bibitem[43]{S1965} S. Wainger, Special trigonometric series in k-dimensions, Mem. Amer. Math. Soc. 59 (1965).
		
		\vspace{-.3cm}
		
		\bibitem[44]{B2023} B. Wei, S. Yang, and Y. Lin, Multiple weighted norm inequalities for multilinear strongly singular integral operators with generalized kernels, Bull. Iranian Math. Soc. 49 (2023), Paper No. 75, 30 pp.
		
		\vspace{-.3cm}
		
		\bibitem[45]{D2010}D. Yang and Y. Zhou, Radial maximal function characterizations of Hardy spaces on RD-spaces and their applications, Math. Ann. 346 (2010), 307--333.
		
		\vspace{-.3cm}
		
		\bibitem[46]{D2011} D. Yang and Y. Zhou, New properties of Besov and Triebel–Lizorkin spaces on RD-spaces, Manuscripta Math. 134 (2011), 59--90.
		
		\vspace{-.3cm}
		
		\bibitem[47]{D20112}D. Yang and Y. Zhou, Localized Hardy spaces $H^1$ related to admissible functions on RD-spaces and applications to Schr\"{o}dinger operators, Trans. Amer. Math. Soc. 363 (2011), 1197--1239.
		
		\vspace{-.3cm}
		
		\bibitem[48]{S2021} S. Yang and Y. Lin, Multilinear strongly singular integral operators with generalized kernels and applications, AIMS Math. 6 (2021), 13533--13551.
		
		\vspace{-.3cm}
		
		\bibitem[49]{S2022} S. Yang and  Y. Lin,  Weighted estimates for multilinear strongly singular Calderón-Zygmund operators with multiple weights. Jordan J. Math. Stat.   15 (2022),  467–495.
		
		\vspace{-.3cm}
		
		\bibitem[50]{S2023} S. Yang and Y. Lin, Multilinear commutators of multilinear strongly singular integral operators with generalized kernels. Georgian Math. J.   30 (2023),  783–801.
	\end{thebibliography}
\end{document}